\documentclass[11pt]{amsart}
\usepackage[margin=1.2in]{geometry}
\usepackage{amssymb}
\usepackage{amsmath}
\usepackage{mathtools}
\usepackage{color}
\usepackage{amsthm}
 \usepackage{lmodern}
\usepackage{tikz}
\usetikzlibrary{matrix}
  \usepackage{caption} 
 \input xy
\xyoption{all}
\usepackage{graphics}
\usepackage{amsfonts}
\usepackage{amssymb}
\usepackage{amscd}
\numberwithin{equation}{section}
\newtheorem{theorem}{Theorem}[section]

\newtheorem{proposition}[theorem]{Proposition}
\newtheorem{lemma}[theorem]{Lemma}
\newtheorem{prop}[theorem]{Proposition}
\theoremstyle{definition}

\newtheorem{remark}[theorem]{Remark}

\newtheorem{Conjecture}[theorem]{Conjecture}

\newcommand\half{\tfrac{1}{2}}

\newcommand\ov{\overline}

\newcommand\be{\beta}

\newcommand\g{\mathfrak g}

\newcommand\h{\mathfrak h}
\newcommand\ha{\widehat{\mathfrak h}}

\newcommand\n{\mathfrak n}

\newcommand\D{\Delta}
\renewcommand\l{\lambda}
\newcommand\Dp{\Delta^+}

\renewcommand\d{\delta}

\renewcommand\a{\alpha}

\newcommand{\Z}{\mathbb Z}
\newcommand\nat{\mathbb N}

\newcommand\s{\sigma}

\newcommand\la{\lambda}

\newcommand\e{\epsilon}
\newcommand\C{\mathbb C}
\newcommand\R{\mathbb R}

\newcommand\si{\sigma}

\renewcommand\ha{\widehat{\mathfrak h}}

\newcommand{\fn}{\mathfrak{n}}

\newcommand{\ZZ}{\mathbb{Z}}

\newcommand{\End}{\mbox{End}}

\newcommand{\vac}{|0\rangle}
\newcommand{\bea}{\begin{eqnarray}}
\newcommand{\eea}{\end{eqnarray}}
\newcommand{\Ws}{W_k^{\min}(\g)}
\newcommand{\Wu}{W^k_{\min}(\g)}

\begin{document}

\title[Spectral flow and application to unitarity]{Spectral flow and application to unitarity of representations of minimal
$W$-algebras}
\author[Victor~G. Kac, Pierluigi M\"oseneder Frajria,  Paolo  Papi]{Victor~G. Kac\\Pierluigi M\"oseneder Frajria\\Paolo  Papi}

\begin{abstract}
Using spectral flow, we provide a proof of  \cite[Theorem 9.17]{KMPR} on unitarity of Ramond twisted non-extremal representations of unitary minimal 
$W$-algebras that does not rely on the still conjectural  exactness of the twisted quantum reduction functor (see Conjecture 9.11 of \cite{KMPR}). When $\g=spo(2|2n)$, $F(4)$, $D(2,1;\frac{m}{n})$, it is also proven that the unitarity of extremal (=massless) representations of the unitary minimal $W$-algebra $\Wu$ in the Ramond sector is equivalent to the unitarity of extremal representations in the Neveu-Schwarz sector.
 \end{abstract}
\maketitle

\section{Introduction}

Unitarity of highest weight representations of infinite-dimensional Lie
algebras and superalgebras has been studied extensively in Mathematics
and Physics literature since 1980's. \par
 In several papers, including \cite{DL}, \cite{KMP}, the problem was put in the framework of vertex operator algebras (VOA) for which the eigenvalues of the zero mode $L_0$ of the conformal vector lie in $\half\Z_{\ge 0}$. In this language, given a conjugate linear involution $\omega$ of a VOA $V$,  a Hermitian form on an irreducible representation $M$ of $V$ is called {\it $\omega$-invariant} if the  adjunction relation 
$$(a^M_n)^*=(-1)^{\D+2\D^2}\omega(a)^M_{-n}$$
holds for any quasi-primary element $a\in V$ of conformal weight $\D(\in \half\Z_{\ge 0})$, and $M$ is called {\it unitary} if the Hermitian form is positive definite.

Most of the (super)algebras for which the problem has been investigated are
either affine Lie  algebras, or arise naturally in the context of a
special class of vertex algebras, called
quantum affine $W$-algebras, obtained by quantum Hamiltonian reduction
from affine vertex algebras. In particular, all superconformal algebras
correspond to the so called minimal $W$-algebras, $\Ws$, where $\g$ is a ``basic'' simple finite-dimensional Lie superalgebra.
\par

In our paper \cite{KMP1} we started a project aiming at the classification of unitary minimal $W$-algebras and their unitary highest weight representations.
The development and the state of art of this project, continued in \cite{KMPR, KMPN4}, is summarized in Section \ref{2} of this paper.
According to the Physicists' terminology, there are two most important classes of representations,  called  the Neveu-Schwarz sector and the Ramond sector, respectively. They correspond to the  untwisted modules and to $\s_R$-twisted modules, where 
$\s_R$ is the automorphism which acts as the identity on  even elements and as  minus the identity on  odd elements.

Regarding the Neveu-Schwarz sector, in \cite{KMP1}, we gave a detailed proof of the classification of  unitary minimal $W$-algebras and their non-extremal (=massive) unitary representations and  found the necessary conditions of unitarity of the extremal (=massless) ones. In \cite{KMPR} and \cite{KMPN4} we  proved necessary conditions for unitarity of highest weight representations  in Ramond sector, but our proof of their unitary  relied on the still conjectural exactness property of the twisted quantum Hamiltonian reduction functor (\cite[Conjecture 9.11]{KMPR}).

In the present paper we provide a proof of unitarity of the massive highest weight representations in Ramond sector  that does not rely on this conjecture (see Theorem \ref{sfnomextremal}). This result  completes the classification of non-extremal unitary representations for all minimal unitary $W$-algebras. 
We also prove that in the cases $\g=psl(2|2)$, $spo(2|2m)$, $D(2,1;a)$, and $F(4)$ unitarity of extremal modules in the Neveu-Schwarz sector is equivalent to that in the Ramond sector (see Theorem \ref{fromNStoR}).

The proofs are based on a functor, described in Section \ref{SF}, between the categories of untwisted modules over a vertex algebra and the Ramond twisted modules. Such functor cannot exist for minimal $W$-algebras when $\g=spo(2|2m+1)$ and $G(3)$ since in these cases the quantum Hamiltonian reduction of a simple module is not simple (\cite[Conjecture 9.11]{KMPR}), while it is simple in the Neveu-Schwarz sector, by Arakawa's theorem. In these cases the proof of unitarity 
of Ramond twisted highest weight massive representations has been already provided by \cite[Corollary 9.5]{KMPR}.

In the paper \cite{Lii2} it is proved that the above  functor is the spectral  flow \cite{Lii}. We thank Haisheng Li for pointing out to us this result.
% but we do not need this fact and its proof would take us far from the goal of this paper.
%Evidence that our spectral flow coincides with that defined in \cite{Lii} is given by the example of the free boson, developed in \cite[Section 6]{KMP1}.

For the basic notions and facts of the vertex algebra theory we refer to \cite{FHL}, \cite{KB}, and for construction of $W$-algebras to \cite{KW1}.
\section{Unitarity of highest weight representations of minimal $W$-algebras}\label{2}
Let $\g$ be a simple finite-dimensional Lie superalgebra, over $\C$, with a reductive even part $\g_{\bar 0}$ and invariant
non-degenerate bilinear form $(\cdot |\cdot)$, with restriction to $\g_{\bar 0}$ non-degenerate.
 Let $\mathfrak s=Span\{e,x,f\}$, where $[e,f]=x, [x,e]=e, [x,f]=-f$ be an $sl_2$ subalgebra of $\g_{\bar 0}$. To the datum
 $(\g,\mathfrak s, k\in\C)$ one associates the universal quantum affine $W$-algebra $W^k(\g,\mathfrak s)$ of level $k$ by the quantum Hamiltonian reduction
 \cite{KRW}, \cite{KW1}. If $k$ is different from the critical level $k_{crit}$ (which we assume), the vertex algebra $W^k(\g,\mathfrak s)$ has a unique simple  quotient, denoted by $W_k(\g,\mathfrak s)$. These VOA are of strong CFT  type \cite{KMP}, hence are generated by quasi-primary elements.
 
 The {\it minimal} $W$-algebras $\Wu$ correspond to a choice of $\mathfrak s$, called minimal, for which the
 $ad\,x$-eigenspace decomposition is of the form 
 \begin{equation}\label{1.1}
 \g=\g_{-1}\oplus\g_{-1/2}\oplus\g_0\oplus\g_{1/2}\oplus\g_1,\quad \text{ where $\g_{-1}=\C f,\,
 \g_{1}=\C e$}.\end{equation}
 We normalize the bilinear form $(\cdot |\cdot)$ by the condition $(x|x)=\half$. Then $k_{crit}=-h^\vee$, where 
 $h^\vee$ is half of the eigenvalue of the Casimir operator on $\g$. The decomposition \eqref{1.1} and the numbers $h^\vee$ are listed in \cite[Tables 1-3]{KW1}.

It is proved in \cite[Proposition 7.2]{KMP1} that for  a non-collapsing level $k\in\R$, any conjugate linear involution $\phi$ of the vertex algebra $\Wu$ is necessarily induced by a conjugate linear involution $\phi$ of $\g$ fixing $\mathfrak s$. (Recall that $k$ is called a collapsing level if $\Ws$ is isomorphic to its affine part.) Moreover, it is proved in \cite[Proposition 8.9]{KMP1} that the vertex algebra $\Ws$
can be  unitary only if the centralizer $\g^\natural$ of $\mathfrak s$ in $\g$ is a semisimple subalgebra of $\g_{\bar0}$, and the conjugate linear  involution $\phi$ is {\sl almost compact}, i.e. it restricts to a compact  involution of $\g^\natural$, and it leaves $\{e,x,f\}$ fixed. We write $\g^\natural=\oplus_i\g^\natural_i$, where $\g^\natural_i$ are simple components of $\g^\natural$.

We prove in  \cite{KMP1} that an almost compact conjugate linear involution of $\g$ exists if and only if $\g$ is from the following lists:
\begin{equation}\label{isunitary}
psl(2|2),\ spo(2|m) \text{ for }m\ge0, D(2,1;a)\text{ for }a\in \R,\ F(4),\ G(3);
\end{equation}
\begin{equation}\label{nonunitary}
sl(2|m)\text{ for }m\ge3,\ osp(4|m)\text{ for }m>2\text{ even},
\end{equation}
and it is essentially unique. 
%Moreover, in these cases $\g$ admits a unique, up to conjugation, minimal $sl_2$-subalgebra $\mathfrak s$. We denote the corresponding minimal universal $W$-algebra of level $k$ by $\Wu$, and its simple quotient by $\Ws$. 
Recall that, for $k\ne k_{crit}$,  the vertex algebra $\Wu$ is of strong CFT type with Virasoro field $L=\sum_{n\in\Z}L_nz^{-n-2}$, and it is strongly and freely generated by the operators $L_n$, $n\in\Z$, and the Fourier coefficients of the primary fields $J^{\{a\}}(z)=\sum_{n\in\Z} J^{\{a\}}_{n}z^{-n-1}$, $a\in\g^\natural$, of conformal weight 1, and $G^{\{u\}}(z)=\sum_{n\in\half+\Z}G^{\{u\}}_nz^{-n-\tfrac{3}{2}}$, $u\in\g_{-1/2}$, of conformal weight $\tfrac{3}{2}$ \cite[Theorems 4.1 and 5.1]{KW1}. The $\lambda$-brackets among these generators are displayed in \cite{KRW, KW1, AKMPP}.\par
When $\g=spo(2|m)$, $m=0,1$, and $2$, the $W$-algebra $\Wu$ is the universal Virasoro, Neveu-Schwarz, and $N=2$ vertex algebra, respectively, for which unitarity  is well understood,  see e.g. \cite{AKMPP, ET1, Gunaydin,  M} for references.
For the remaining cases, we proved in \cite[Proposition 8.19]{KMP1} that, for $k\ne k_{crit}$, the minimal $W$-algebra $\Ws$ is not unitary for $\g$ from the list \eqref{nonunitary}, except when $\g=sl(2|m)$, $m\ge3$, and the level is the collapsing level $k=-1$. Furthermore, we proved in \cite[Corollary 11.2]{KMP1} that, for $\g$ from the list \eqref{isunitary}, the vertex algebra $\Ws$ is non-trivial   unitary for $k\ne k_{crit}$ if and only if $k$ lies in the {\sl unitary range}, given in the  following Table \ref{tabel0}, where we also display the critical values of $k$ and the collapsing values   $k_0$ for which $\dim W_{k_0}^{\min}(\g)=1$:
\renewcommand{\arraystretch}{1.4}
\begin{center}
\begin{tabular}{c | c| c |c  | c}
$\g$&
unitary range&
$k_{crit}$& $k_0$ & $\chi_i$ \\
\hline
$psl(2|2)$&$-(\nat+1)$&$0$&$-1$ &$-1$\\
\hline
$spo(2|3)$&$-\tfrac{1}{4}(\nat+2)$&$-\half$&$-\half$&$-2$\\
\hline
$spo(2|m),\,m\ge4$&$-\tfrac{1}{2}(\nat+1)$&$\tfrac{m}{2}-2$&$-\half$&$-1$\\
\hline
$D(2,1;\tfrac{m}{n})$&$-\tfrac{mn}{m+n}\nat,\ m,n\in\nat\text{ coprime},\, (m,n)\ne(1,1)$&$0$& \text{none}&$-1,-1$\\
\hline
$F(4)$&$-\tfrac{2}{3}(\nat+1)$&$2$&$-\tfrac{2}{3}$&$-1$\\
\hline
$G(3)$&$-\tfrac{3}{4}(\nat+1)$&$\tfrac{3}{2}$&$-\tfrac{3}{4}$&$-1$
\end{tabular}
 \captionof{table}{\label{tabel0}}
\end{center}
 In our paper \cite{KMP1}, we also studied unitarity of irreducible highest weight $\Wu$-modules $L^W(\nu,\ell)$ (to be described in detail in Section \ref{5}), where $\g$ is one of the Lie superalgebras from Table \ref{tabel0} (with the exception of $spo(2|m)$, $m\le 2$, for which the answer is well-known), and $k$ lies in the unitary range. These modules are parametrized by pairs $\nu\in (\h_\R^\natural)^*$ and $\ell\in\R$. The following are necessary conditions for  unitarity of $L^W(\nu,\ell)$: 
 \begin{enumerate}
 \item[(NS1)] the affine levels $M_i(k)$  for $\g^\natural_i$,  explicitly displayed in \cite[Table 2]{KMP1}, are non-negative integers;
 \item[(NS2)] $\nu\in P^+_k=\{\text{dominant integral weights for }\g^\natural\text{ such that }\nu(\theta_i^\vee)\le M_i(k)\}$ where $\theta_i$ are highest roots of $\g_i^\natural$.
 \item[(NS3)]$\ell\ge A(k,\nu)$, where $A(k,\nu)$ is defined in \cite[formula (8.11)]{KMP1}, and $\ell=A(k,\nu)$ if $\nu$ is an extremal weight (i.e. $\nu(\theta_i^\vee)>M_i(k)+\chi_i$ for some $i$, $\chi_i$ being displayed in Table \ref{tabel0}. \end{enumerate}
 \begin{theorem}\cite{KMP1} The highest weight modules
\begin{equation}\label{hwu}
\left\{L^W(\nu,\ell)\mid \nu\in P^+_k,\nu\text{\rm\ non-extremal, }\ell\ge A(k,\nu)\right\}
\end{equation}
are unitary.
\end{theorem}
\begin{Conjecture}\label{NS}
The set of unitary highest weight modules  over $\Wu$ is the union of \eqref{hwu} and 
\begin{equation*}
\left\{L^W(\nu,A(k,\nu))\mid \nu\in P^+_k,\nu\text{ extremal}\right\}.
\end{equation*}
\end{Conjecture}
 
The Ramond case was dealt with in the paper \cite{KMPR}. We had to introduce a suitable notion of extremality, called {Ramond extremality} (cf. \cite[(9.3)]{KMPR}) and a constant $A_R(k,\nu)$ (cf. \cite[(6.31)]{KMPR}  in place of $A(k,\nu)$. 
There are two types  of these modules satisfying the necessary conditions of unitarity:
\begin{enumerate}
\item[(R1)] the modules $L_R^W(\nu,\ell)$ with $\nu\in P^+_k$ not Ramond extremal and $\ell\ge  A_R(k,\nu)$;
\item[(R2)] the modules $L_R^W(\nu,\ell)$ with the weight $\nu\in P^+_k$  Ramond extremal,
in which case $
\ell=A_R(k,\nu)$.
\end{enumerate}
In the next sections  we shall build the spectral flow functor between Neveu Schwarz and Ramond sectors for minimal $W$-algebras. Using this functor, we shall prove that the modules 
listed in (R1) are indeed unitary. In \cite{KMPR} we proved the result  assuming true the exactness of the twisted quantum Hamiltonian reduction functor which, in the untwisted case,  holds by the work of  Arakawa. In case (R2) we don't know how to prove unitarity. Summing up, we have the following theorem and conjecture.
\begin{theorem} \cite{KMPR} The highest weight Ramond twisted modules
\begin{equation}\label{hwur}
\left\{L^W_R(\nu,\ell)\mid  \nu \text{ \rm is not  Ramond extremal}, \ell\ge A_R(k,\nu)\right\}
\end{equation}
are unitary.
\end{theorem}
\begin{Conjecture}\label{R}
The set of unitary highest weight modules over  $\Wu$ is the union of \eqref{hwur} and
\begin{equation*}
\left\{L^W_R(\nu,A(k,\nu))\mid\nu \text{ is   Ramond extremal}\right\}.
\end{equation*}
\end{Conjecture}

\section{A functor between categories of positive energy twisted modules}\label{SF}

In Section \ref{4} we will build up a functor between positive energy $\Wu$-modules mapping  untwisted highest weight modules to Ramond twisted highest weight modules.
This functor is obtained by applying twice   the  construction  of the contragredient module. It is shown in \cite[Section 2]{Lii2} that this functor  coincides with the spectral flow defined in \cite{Lii}.
To keep the paper as self-contained as possible and to take care of the modifications needed to work with  conjugate dual modules rather than  contragredient ones, we provide in this section some details of this construction.

Let $V$ be a vertex  algebra that
admits a conformal (=Virasoro) vector $L$ with $L_0$ acting diagonally with real eigenvalues (=conformal weights).  Let 
$$
V=\oplus_{\D\in \R}V_\D
$$
be the corresponding eigenspace decomposition.
We make the further assumption that  $V_0=\C\vac$.
This implies that $V_1$ is a Lie superalgebra under the bracket
$$
[a,b]=a_{(0)}b,
$$
 that admits an invariant bilinear form $\be$ given by 
$$
a_{(1)}b=\be(a,b)\vac.
$$

Let $f$ be a parity preserving diagonalizable automorphism of $V$ with modulus one eigenvalues such that $f(L)=L$. If $\gamma\in\R$ set
$V^{[\gamma]}$ to be the $e^{2\pi\sqrt{-1}\gamma}$-eigenspace of $f$ so that
$$
V=\oplus_{[\gamma]\in\R/\ZZ}V^{[\gamma]}
$$
is a vertex algebra grading and $L_0(V^{[\gamma]})\subseteq V^{[\gamma]}$.

Recall that 
an $f$--\emph{twisted module} for $V$ is a vector superspace $M$ and a parity preserving
linear map from $V$ to the space of $\End M$--valued 
$f$--\emph{twisted quantum fields} 
$ V^{[\eta_a]} \ni a\mapsto Y_M (a,z) = \sum_{m \in [\gamma_a]} a^M_{(m)}
z^{-m-1}$ (i.e. $a^M_{(m)} \in \End M$ and $a^M_{(m)} v=0$ for
each $v \in M$ and $m \gg 0$), such that (2.12) and (2.13) of \cite{KMP} hold.

Since $V^{[\gamma]}$ is $L_0$--invariant, we have its eigenspace decomposition $V^{[\gamma]}=\oplus_\D V_\D^{[\gamma]}$, and we will write for $a\in V_{\Delta_a}^{[\gamma]}$,
$$
Y_M(a,z)=\sum_{n\in [\gamma-\D_a]}a^M_nz^{-n-\D_a}.
$$

An $f$--twisted $V$--module $M$ is called an \emph{$L$-positive
  energy  $V$--module} if $M$ has an
$\R$--grading $M =\oplus_{{j \geq 0}} M_j$ such that
\begin{equation}
  \label{eq:2.39}
  a^M_n M_j \subseteq M_{j - n}, \ a\in V^{[\gamma]}_{\D_a} .
\end{equation}
The subspace ~$M_0$ is called the
\emph{minimal energy subspace}.  Then,
\begin{equation}
  \label{eq:2.40}
  a^M_n M_0 = 0 \hbox{  for  } n>0 \hbox{  and  }a^M_0 M_0
  \subseteq M_0 \, .
\end{equation}

Set
$$p(a)=\begin{cases}0&\text{ if $a\in V_{\bar{0}}$},\\ 1 &\text{ if $a\in V_{\bar{1}}$, }
\end{cases}
$$
Note that we will regard  $p(a)$ as an integer, not as a residue class.
 If $a\in V_{\D_a}$, set for $t\in\R$,
\begin{align}\label{(-1)L0}
(-1)^{tL_0}a&=e^{\pi\sqrt{-}1t\D_a}a,\quad\s^{t}(a)=e^{\pi\sqrt{-}1tp(a)}a.
\end{align}
%where $\s(a)$ is defined in \eqref{s}.
%
\begin{lemma}\label{twistAz}Let $g$ be a diagonalizable parity preserving conjugate  linear  operator on $V$ with modulus $1$ eigenvalues.
Then  
\begin{equation}\label{congsign}
g Y(a,z)g^{-1} b=p(a,b)Y(g (a),-z)b
\end{equation}
if and only if (with the notation established in \eqref{(-1)L0})
 \begin{equation}\label{gamma}\phi=g (-1)^{-L_0} \s^{-1/2} \end{equation} is  a conjugate linear automorphism of $V$.
\end{lemma}
\begin{proof}
 Assume that $g$ satisfies \eqref{congsign}. This means that
\begin{equation}\label{ggggg}
 g(a_{(n)}b)=(-1)^{n+1} p(a,b)g(a)_{(n)}g(b).
\end{equation}

 Then
  \begin{align}\label{23}\phi(a)_{(n)}\phi(b)&=g((-1)^{-L_0} \s^{-1/2}(a)_{(n)}g((-1)^{-L_0} \s^{-1/2})(b)\\
&=e^{\pi\sqrt{-}1(\D_a+\D_b)}e^{\pi/2\sqrt{-}1(p(a)+p(b))}g(a)_{(n)}g(b).\notag
\end{align}
By \eqref{ggggg}, substituting in \eqref{23}, and noting 
that $p(a)+p(b)+2p(a)p(b)=p(a_{(n)}b)\mod 4\Z$, we obtain
\begin{align*}\phi(a)_{(n)}\phi(b)&=e^{\pi\sqrt{-}1(\D_a+\D_b)}e^{\pi/2\sqrt{-}1(p(a)+p(b))}(-1)^{n+1} p(a,b)g(a_{(n)}b)\\
&=e^{\pi\sqrt{-}1\D_{a_{(n)}b}}e^{\pi/2\sqrt{-}1(p(a)+p(b)+2p(a)p(b))}g(a_{(n)}b)\\
&=ge^{-\pi\sqrt{-}1\D_{a_{(n)}b}}e^{-\pi/2\sqrt{-}1p(a_{(n)}b)}(a_{(n)}b)=\phi(a_{(n)}b).
\end{align*}
Reversing the argument we obtain the converse statement. 
\end{proof}

Note that Lemma \ref{congsign} implies that $g(\vac)=\vac$,  $g(\partial a)=-\partial g(a)$ and that $L':=g(L)$ is a conformal  vector for $V$. Indeed, $g(\vac)=\phi(\vac)=\vac$;  if $a\in V_\D$, then
$$
\partial g(a)=\partial \phi \s^{1/2} (-1)^{L_0} a=\phi(e^{\pi\sqrt{-}1(\D_a+\frac{1}{2}p(a))}\partial a=-\phi(e^{\pi\sqrt{-}1(\D_{\partial a}+\frac{1}{2}p(\partial a))}\partial a=-g(\partial a);
$$
and, since $g(L)=\phi(L)$, we have
$$
[g(L)_\la g(L)]=\phi([L_\la L])=\partial \phi( L)+2\phi(L)+\frac{ \la^3 }{12}\ov c\vac=\partial g( L)+2g(L)+\frac{ \la^3 }{12}\ov c\vac.
$$

 Let $g$ be a diagonalizable parity preserving conjugate linear  operator on $V$  satisfying \eqref{congsign}.
Define $A(z):V_\D\to z^{-2\D} V[[z]]$ by 
\begin{equation}\label{AZ}
A(z)v=ge^{-zL_1}  z^{-2L_0}v.
\end{equation}

Using  the results of \cite[\S\ 4.9]{KB}, one deduces from Lemma \ref{twistAz}, as in Lemma 3.3 of \cite{KMP}, that
 \begin{equation}\label{conjAzzt}
p(a,b)A(w)Y(a,z)A(w)^{-1}b=i_{w,z}Y\left( A(z+w)a,\frac{-z}{(z+w)w}\right)b.
\end{equation}

If $M$ is a $L$-positive energy $f$-twisted $V$-module, we set 
$$
M^\dagger=\oplus_{n}M^\dagger_n,
$$
where $M^\dagger_n$ is the  conjugate linear dual of $M_n$. Recall that $M^\dagger_n$ is the set of additive maps $f:M_n\to \C$ such that $f(\alpha m)=\ov{\alpha} f(m)$ for $m\in M_n$ and $\alpha\in\C$.

{\sl Assume now that  the conformal weights for $L$ are in $\half \ZZ_+$ and that $\phi(L)=L$.} Choose an even element $h\in V_1$ such that $\phi(h)=-h$ and $f(h)=h$. Assume that $h_0$ acts semisimply on $V$ with real eigenvalues.
For $t\in\R$, set
$$L\langle t h\rangle=L+t\partial h.$$ Since $V_\D=\{0\}$ for $\D<0$, we have
$
[h_\l h]=\l\be(h,h)\vac.
$
This implies that, for $t\in\R$
$$
[L\langle t h\rangle_\l L\langle t h\rangle]=\partial L\langle t h\rangle+2L\langle t h\rangle+\la^3\frac{c-12t^2\be(h,h)}{12},
$$
so $L\langle t h\rangle$ is a Virasoro vector for $V$. 
To clarify notation, we set  $V_{\D(t)}$ to be the eigenspace for $L\langle t h\rangle$ corresponding to the conformal weight $\D(t)$. If $M$ is a $f$-twisted module and $a\in V^{[\gamma]}_{\D_a}$, we will write
$$
Y_M(a,z)=\sum_{n\in[\gamma-\D_a]}a^M_nz^{-n-\D_a},
$$
while, if $a\in V^{[\gamma]}_{\D_a(t)}$, we write the mode expansion with respect to  $L\langle t h\rangle$ as
$$
Y_M(a,z)=\sum_{n\in[\gamma-\D_a(t)]}a^M(n,t)z^{-n-\D_a(t)}.
$$
Note that, since $L\langle t h\rangle_{(1)}L\langle t h\rangle=2L\langle t h\rangle=L_{(1)}L\langle t h\rangle$, we have
$$
L\langle t h\rangle(n,t)=L\langle t h\rangle_n=L_n-t(n+1)h_n.
$$
In particular 
\begin{equation}\label{Lth0}
L\langle t h\rangle(0,t)=L\langle t h\rangle_0=L_0-th_0.
\end{equation}
Since $[f,h_0]=0$, we can write 
\begin{equation}\label{fffff}V^{[\gamma]}=\bigoplus_{\gamma'\in\R} V^{[\gamma],\gamma'},\end{equation}
where $V^{[\gamma],\gamma'}$ is the $\gamma'$-eigenspace for $h_0$. Since $f(L)=L$, we have that 
$$V^{[\gamma],\gamma'}=\bigoplus_\D V^{[\gamma],\gamma'}_\D.$$
By \eqref{Lth0}, if $a\in V^{[\gamma_a],\gamma'_a}_{\D_a}$,  then $\D_a(t)=\D_a-t\gamma'_a$ and
$$
a^M(n,t)=a^M_{n-t\gamma'_a}.
$$

Also, by \eqref{Lth0}, $L\langle t h\rangle(0,t)=L\langle t h\rangle_0$ acts semisimply on $V$ with real eigenvalues, so that we can define
 $$A(z,th)=ge^{-zL\langle t h\rangle_1}  z^{-2L\langle t h\rangle_0},
 $$
 as in \eqref{AZ}, with $L=L\langle t h\rangle$ and $g=\phi\s^{1/2}(-1)^{L\langle t h\rangle_0}$.

\begin{prop}\label{functor}
Let $M$ be an $f$-twisted $L\langle t h\rangle$-positive energy $V$-module. Set, for $m^\dagger\in M^\dagger, m\in M$ and $a\in V$, 
\begin{equation}\label{mmmmm}
(Y_{M^\dagger} (a,z)m^\dagger)(m)=m^\dagger(Y_M(A(z,th)a,z^{-1})m).
\end{equation}
Define
$$f\langle t\rangle=\phi^{-1}e^{-4\pi \sqrt{-}1th_0}f\phi.$$
 Then
$Y_{M^\dagger}  (a,z)$ is a $f\langle t\rangle$-twisted quantum field and the map $a\mapsto Y_{M^\dagger}  (a,z)$ gives  $M^\dagger$ the structure of an 
$f\langle t\rangle$-twisted $L\langle 3t h\rangle$-positive energy $V$-module.
\end{prop}
\begin{proof}
Recall that  $V=\bigoplus\limits_{\gamma}V^{[\gamma]}$ is the grading induced by $f$. Write $V=\bigoplus\limits_{\gamma^{(t)}}V^{[\gamma^{(t)}]}$ for the eigenspace decomposition with respect to $f\langle t\rangle$. 
Clearly, with notation as in \eqref{fffff}, we have 
\begin{equation}\label{df}V^{[\gamma^{(t)}]}=\bigoplus _{\gamma^{(t)}=-\gamma_1+2t\gamma'}\phi^{-1}(V^{[\gamma_1],\gamma'}).
\end{equation}
Note also that, since $\phi(L)=L$, we have   $L\langle s\rangle_0=\phi^{-1}\circ L\langle -s\rangle_0$, so $V^{[\gamma^{(t)}]}$ is $L\langle s\rangle_0$-stable and 
$$V^{[\gamma^{(t)}]}_{\D(s)}=\bigoplus _{\gamma^{(t)}=-\gamma_1+2t\gamma'}\phi^{-1}(V^{[\gamma_1],\gamma'}_{\D(-s)}).
$$ 
In particular, if    $a\in \phi^{-1}(V^{[\gamma_1],\gamma'})_{\D_a}= \phi^{-1}(V^{[\gamma_1],\gamma'}_{\D_a})$ then 
\begin{equation}\label{das}\D_a(s)=\D_a+s\gamma'.\end{equation}
By \eqref{mmmmm},
\begin{equation}\label{s1}(Y_{M^\dagger}(a,z)m^\dagger)(m)=
\sum_{r\in\Z_+}(\tfrac{(-1)^r}{r!}m^\dagger ( Y_M(g({L\langle t h\rangle_1^r}a), z^{-1})m)z^{-2\D_a-2t\gamma'+r}.
\end{equation}
We have $a=\phi^{-1}(b)$ with $b\in V^{[\gamma_1],\gamma'}_{\D_a}$ so
$$
\begin{aligned}
&g(L\langle t h\rangle_1^ra)=\phi\s^{1/2}(-1)^{L\langle t h\rangle_0}L\langle t h\rangle_1^r\phi^{-1}(b)=\s^{-1/2}(-1)^{-L\langle t h\rangle_0}\phi\phi^{-1}(L\langle-t\rangle^r_1b)\\
&=\s^{-1/2}(-1)^{-L\langle t h\rangle_0}(L\langle-t\rangle^r_1b)\in V^{[\gamma_1],\gamma'}_{\D_a-r},
\end{aligned}
$$
hence we can write
\begin{equation}\label{s2}
 Y_M(g(L\langle t h\rangle_1^ra), z)=\sum_{n\in[\gamma_1-\D_a+t\gamma']}g(L(h)_1^ra)^M(n,t)z^{-n-\D_a+t\gamma'+r},
\end{equation}
so, substituting \eqref{s2} into \eqref{s1}, we get
$$
\begin{aligned}
&(Y_{M^\dagger}(a,z)m^\dagger)(m)\\&=
\sum_{r,n\in[\gamma_1-\D_a+t\gamma']}(\tfrac{(-1)^r}{r!}m^\dagger (g(L\langle t h\rangle_1^ra)^M(n,t)m)z^{n-\D_a -3t\gamma'}\\
&=\sum_{r,n\in[-\gamma_1-t\gamma'+\D_a]}(\tfrac{(-1)^r}{r!}m^\dagger(g(L\langle t h\rangle_1^ra)^M(-n,t)m)z^{-n-\D_a(3t)}\\
&=\sum_{r,n\in[-\gamma_1+2t\gamma'-\D_a(3t)]}(\tfrac{(-1)^r}{r!}m^\dagger (g(L\langle t h\rangle_1^ra)^M(-n,t)m)z^{-n-\D_a(3t)}\\
&=\sum_{r,n\in[-\gamma_1+2t\gamma'-\D_a(3t)]}(\tfrac{1}{r!}m^\dagger ((g(L\langle t h\rangle)_1^rg(a))^M(-n,t)m)z^{-n-\D_a(3t)},
\end{aligned}
$$
where, in the last equalities, we used the assumption that $\D_a\in\half\ZZ$.
Writing  explicitly 
$$
Y_{M^\dagger}(a,z)=\sum_{n\in [\gamma-\D_a(3t)]}a^{M^\dagger}(n,3t)z^{-n-\D_a(3t)},
$$
we find that \begin{equation}\label{undagger}
(a^{M^\dagger}(n,3t)m^\dagger)(m) =m^\dagger((e^{g(L\langle t h\rangle)_1}g(a))^M(-n,t)m).
\end{equation}
In particular
\begin{equation}\label{ispe}
a^{M^\dagger}(n,3t)M^\dagger_j\subseteq M^\dagger_{j-n}.
\end{equation}
This proves that  $Y^{M^\dagger}$ is indeed an  $f\langle t\rangle$-twisted quantum field.

Next observe that
\begin{equation}\label{vacdagger}
( \vac^{M^\dagger}_{(n)}m^\dagger)(m)=(\vac^{M^\dagger}(n+1,3t)m^\dagger)(m)=m^\dagger(\vac^M_{-n-1}m)=\d_{-n-1,0}m^\dagger(m),
\end{equation}
hence (2.12) of \cite{KMP} holds for $M^\dagger$.
The fact that $M^\dagger$ is $L\langle3t\rangle$-positive energy follows readily from \eqref{ispe}. \par
We now prove the Borcherds identity for $M^\dagger$.  We confine ourselves to point out  the  modifications to  the proof of Proposition 3.6 of \cite{KMP} (which is inspired by   Section 5 of \cite{FHL}) needed to accommodate our setting. 

Let $a\in V^{[\gamma_a^{(t)}]}, b\in V^{[\gamma_b^{(t)}]}$, $m\in M$, $m^\dagger \in M^\dagger$. We have to prove that 
\begin{align}
&Res_{u}( Y_{M^\dagger}(Y(a,u)b,w)i_{w,u}(w+u)^ku^nw^q{m^\dagger})(m)\notag\\
&=Res_z( Y_{M^\dagger}(a,z)Y_{M^\dagger}(b,w)i_{z,w}z^k(z-w)^nw^q{m^\dagger})(m)\\
&-p(a,b)Res_z( Y_{M^\dagger}(b,w)Y_{M^\dagger}(a,z)i_{w,z}z^k(z-w)^nw^q{m^\dagger})(m)\notag
\end{align}
for all $n\in\Z$, $k\in[\gamma_a^{(t)}]$, $q\in[\gamma_b^{(t)}]$. Here, as usual, $i_{w,u}$ means expanding in the domain $|w|>|u|$.

 By \eqref{df}, we can assume that $\phi(a)\in V^{[\gamma_1],\gamma'},  \phi(b)\in V^{[\eta_1],\eta'}$, where $\gamma_a^{(t)}=-\gamma_1+2t\gamma',$ $ 
\gamma_b^{(t)}=-\eta_1+2t\eta'$.
Since
 \begin{align*}
( Y_{M^\dagger}(a,z)Y_{M^\dagger}(b,w){m^\dagger})(m)&=m^\dagger(Y_M(A(w,th)b,w^{-1})Y_M(A(z,th)a,z^{-1})(m),\\
( Y_{M^\dagger}(b,w)Y_{M^\dagger}(a,z){m^\dagger})(m)&=m^\dagger(Y_M(A(z,th)a,z^{-1})Y_M(A(w,th)b,w^{-1})(m),\\
( Y_{M^\dagger}(Y(a,u)b,w){m^\dagger})(m)&=m^\dagger(Y_M(A(w,th)Y(a,u)b,w^{-1})(m),
\end{align*}
we have to prove that
\begin{align*}
&Res_{u}m^\dagger(Y_M(A(w,th)Y(a,u)b,w^{-1})(m) i_{w,u}(w+u)^ku^nw^q)\\
&=Res_z{m^\dagger}(Y_M(A(w,th)b,w^{-1})Y_M(A(z,th)a,z^{-1})(m) i_{z,w}z^k(z-w)^nw^q)\\
&-p(a,b)Res_z( {m^\dagger}(Y_M(A(z,th)a,z^{-1})Y_M(A(w,th)b,w^{-1})(m)i_{w,z}z^k(z-w)^nw^q).
\end{align*}
Hence we need to check that
\begin{align}
&Res_{u}(Y_M(A(w,th)Y(a,u)b,w^{-1}) i_{w,u}(w+u)^ku^nw^q)\notag\\
&=Res_z(Y_M(A(w,th)b,w^{-1})Y_M(A(z,th)a,z^{-1})i_{z,w}z^k(z-w)^nw^q)\label{needtocheck}\\
&-p(a,b)Res_z(Y_M(A(z,th)a,z^{-1})Y_M(A(w,th)b,w^{-1}) i_{w,z}z^k(z-w)^nw^q).\notag
\end{align}

Changing variables in the Borcherds identity   for $Y_M$ we obtain, for all $n\in\Z$, $p\in[\gamma_a]$, $d\in[\gamma_b]$,
\begin{align*}
&Res_{u}Y_M(Y(a,u^{-1})b,w^{-1})i_{w^{-1},u^{-1}}(w^{-1}+u^{-1})^pu^{-n-2}w^{-d}\\
&=Res_u(Y_M(a,u^{-1})Y_M(b,w^{-1})i_{u^{-1},w^{-1}}u^{-p-2}(u^{-1}-w^{-1})^nw^{-d})\\
&-p(a,b)Res_u(Y_M(b,w^{-1})Y_M(a,u^{-1})i_{w^{-1},u^{-1}}u^{-p-2}(u^{-1}-w^{-1})^nw^{-d}),\notag
\end{align*}
which is equivalent to
\begin{align}
&Res_{u}Y_M(Y(a,u^{-1})b,w^{-1})i_{u,w}(w+u)^{p}u^{-n-2-p}w^{-d-p}\notag\\
&=Res_{u}(Y_M(a,u^{-1})Y_M(b,w^{-1})i_{w,u}u^{-p-n-2}(w-u)^nw^{-d-n})\label{Jacres}\\
&-p(a,b)Res_{u}(Y_M(b,w^{-1})Y_M(a,u^{-1})i_{u,w}u^{-p-2-n}(w-u)^nw^{-d-n})\notag.
\end{align}
Now remark  that 
\begin{align*}
A(z,t h)a&=\phi\s^{1/2}(-1)^{L\langle t\rangle_0}e^{-zL\langle t\rangle_1}  z^{-2L\langle t\rangle_0}a\\
&=\sum_{r\in\Z_+}\frac{(-1)^r}{r!}e^{-\pi\sqrt{-1}(\D_a(t)+\frac{1}{2}p(a))}(L\langle -t\rangle)^r_1  \phi(a)z^{r-2\D_a(t)}
\end{align*}
hence, setting $C_r(a)=\frac{(-1)^r}{r!}e^{-\pi\sqrt{-1}(\D_a(t)+\frac{1}{2}p(a))}(L\langle -t\rangle)^r_1  \phi(a)$, we have  
\begin{equation}\label{expAz}Y_M(A(z,t h)a, z^{-1})=\sum_{r\in\Z_+} Y_M(C_r(a),z^{-1})z^{r-2\D_a(t)}.\end{equation}

Consider now the following expression, where  $m\in [\gamma_1]$ and $v\in[\eta_1]$
\begin{equation}\label{f1}
\begin{aligned} 
&Res_{z}(Y_M(A(z,t h)a,z^{-1})Y_M(A(w,t h)b,w^{-1})\times\\
&\quad\quad i_{w,z}z^{-m-n-2+2\D_a(t)}(w-z)^nw^{-v-n+2\D_b(t)})\\
&-p(a,b)Res_{z}(Y_M(A(w,t h)b,w^{-1})Y_M(A(z,t h)a,z^{-1})\times\\
&\quad\quad i_{z,w}z^{-m-2-n+2\D_a(t)}(w-z)^nw^{-v-n+2\D_b(t)}).
\end{aligned}
\end{equation}
By \eqref{Jacres} and \eqref{expAz}, we can rewrite \eqref{f1} as 
\begin{align*}
&\sum_{r,s}Res_{z}(Y_M(C_r(a),z^{-1})Y_M(C_s(b),w^{-1})i_{w,z}z^{-m-n-2+r}(w-z)^nw^{-v-n+s})\\
&-p(a,b)\sum_{r,s}Res_{z}(Y_M(C(b),w^{-1})Y_M(C_r(a),z^{-1})i_{z,w}z^{-m-2-n+r}(w-z)^nw^{-v-n+s})\\
%%%%%
&=\!\!\sum_{r,s}\!Res_{z}(Y_M(Y(C_r(a),z^{-1})C(b),w^{-1})i_{z,w}(w+z)^{m-r}z^{-n-2-m+r}w^{-v+s-m+r}=\\
%%%
&\!\!\sum_{r}\!Res_{z}(Y_M(Y(C_r(a),z^{-1})A(w,t h)b,w^{-1})\times\\
&i_{z,w}(w+z)^{m-r-2\D_a(t)+2\D_a(t)}z^{-n-2-m+r-2\D_a(t)+2\D_a(t)}w^{-v-m+r-2\D_a(t)+2\D_a(t)+2\D_b(t)}=\\
&Res_{z}(Y_M(i_{z,w}Y(A(\frac{wz}{w+z},t h)a,z^{-1})A(w,t h)b,w^{-1})\times\\
%%%
&(w+z)^{m-2\D_a(t)}z^{-n-2-m+2\D_a(t)}w^{-v-m+2\D_a(t)+2\D_b(t)}.
\end{align*}
By a change of variables, \eqref{f1} turns  into  the right hand side of \eqref{needtocheck}. Indeed, 
setting  $k=-n-2-m+2\D_a(t), q=-v-n+2\D_b(t)$, we have
\begin{align*}
&Res_{z}(Y_M(A(z, t h)a,z^{-1})Y_M(A(w, t h)b,w^{-1})i_{w,z}z^{k}(w-z)^{n}w^{q})\\
&-p(a,b)Res_{z}(Y_M(A(w, t h)b,w^{-1})Y_M(A(z, t h)a,z^{-1})i_{z,w}z^{k}(w-z)^{n}w^{q})\\
&=Res_{z}(Y_M(i_{z,w}Y(A(\frac{wz}{w+z}, t h)a,z^{-1})A(w, t h)b,w^{-1})(w+z)^{-k-n-2}z^{k}w^{q+k+2n+2}.
\end{align*}
Note that $k\in [\gamma_a^{(t)}]$ and $q\in [\gamma_b^{(t)}]$. Hence  
\eqref{needtocheck} turns into
 \begin{align*}
&Res_{z}(i_{z,w}Y_M(Y(A(\frac{wz}{w+z}, t h)a,z^{-1})A(w,t h)b,w^{-1})(w+z)^{-k-n-2}z^{k}w^{q+k+2n+2}\\
&=-p(a,b)(-1)^nRes_{z}(Y_M(A(w,t h)Y(a,z)b,w^{-1}) i_{w,z}(w+z)^kz^nw^q).
\end{align*}
From here on the proof follows verbatim that of Proposition 6.3  \cite{KMP} using \eqref{conjAzzt}.
\end{proof}

\section{Spectral flow for minimal $W$-algebras.}\label{4}

We now specialize to $W^k_{\min}(\g)$ with $\g$ from list \eqref{isunitary}.
In particular we have that 
\begin{enumerate}
\item $\Wu$ admits a Virasoro vector $L$ with $L_0$ acting diagonally. Moreover the gradation of $V$ by conformal weights is
$$\Wu=\bigoplus_{n\in\half\ZZ_+}\Wu_n,$$
with $\Wu_0=\C\vac$ and $\Wu_{\frac{1}{2}}=\{0\}$.
\item $\Wu_{\ov 0}=\oplus_{n\in\ZZ_+}\Wu_n$ and $\Wu_{\ov 1}=\oplus_{n\in\half+\ZZ_+}\Wu_n$.

\end{enumerate}
%We consider only the cases when  all odd roots of  $\g$ are isotropic. We also disregard the case $\g=sl(2|m)$ since in this case $k=-1$ and $\Ws$ is the vertex algebra of one free boson. The remaining Lie superalgebras we are considering are listed in Table \ref{thetahalfisnotroot}.\item  There is a conjugate linear involution $\omega$ on $\Wu$ and  a $\omega$-invariant Hermitian form $(\cdot,\cdot)$ on $\Wu$. We normalize the form by setting $(\vac,\vac)=1$.

In \cite[Table 1]{KMP1}, we chose a special set of simple roots for $\g$, denoted by  $S$. It has the property that, if $\D^\natural$ is the set of roots for $(\g^\natural,\h^\natural)$,  then $S\cap\D^\natural$ is a set of simple roots for $\g^\natural$.
Let  
$\g^\natural=\n^\natural_-\oplus\h^\natural\oplus\n^\natural_+$ be the corresponding triangular decomposition and let $\D^\natural_+$ be the corresponding set of positive roots.
Recall that $\g^\natural$ is a semisimple Lie algebra and write $\g^\natural=\oplus_{i=1}^{r_0}\g^\natural_{i}$ for its decomposition into simple ideals. Actually, $\g^\natural$ is simple  except when $\g=D(2,1;a)$, so $r_0=1$ except when $\g=D(2,1;a)$ where $r_0=2$.
We let $\omega^j_i$ be the fundamental weights of $\g^\natural_j$ corresponding to our choice of simple roots. Likewise let $\theta^\natural_i$ be the highest root of $\g^\natural_i$. Let $\rho_R$ be as in Table \ref{thetahalfisnotroot}. These data  appear also in Tables 2,3 from \cite{KMPR}, where one more choice $\rho_R=\omega_3^1 $ in type $F(4)$ appears.

 We need to treat this extra case with a special argument (cf. Lemma \ref{42}). 

\renewcommand{\arraystretch}{1.5}
\begin{table}[h]
\begin{tabular}{c | c| c |c |c}
$\g$&$psl(2|2)$&
$spo(2|2r)$& $D(2,1;\tfrac{m}{n})$&$F(4)$\\
\hline
$\rho_R$&$\omega^1_1$&$\omega^1_r,\omega^1_{r-1}$&$\omega_1^1,\omega_1^2$
&$\omega^1_1$
\end{tabular}
\captionof{table}{The algebras $\g$ and the weights $\rho_R$.\label{thetahalfisnotroot}}
\end{table}
Denote by $(\cdot|\cdot)$ the invariant bilinear form on $\g$ described explicitly in \cite[Table 1]{KMP1}. 
Identify $\h^\natural$ with $(\h^\natural)^*$ using  $(\cdot|\cdot)$ and define $h^R\in\h^\natural$ by
\begin{equation}\label{hoR}
h^R=4\frac{\rho_R}{(\theta_i^\natural|\theta_i^\natural)}\text{ if $\rho_R=\omega^i_j$.}
\end{equation}

\begin{lemma}\label{hRexixts} $ad(h^R)_{|\g^\natural}$ has eigenvalues in $2\Z$ and  $ad(h^R)_{|\g_{-{1}/{2}}}$ has eigenvalues in $2\Z+1$.
\end{lemma}
\begin{proof}
Note that 
\begin{equation}\label{eigenhR}
\a(h^R)\in\{\pm2,0\} {\rm\ if }\ \g_{\a}\subset \g^\natural, \ \a(h^R)\in\{\pm1\} {\rm \ if }\ \g_{\a}\subset \g_{-1/2}.
\end{equation}
\end{proof}

\begin{remark}
Note that Lemma \ref{hRexixts} does not hold if $\g$ is $spo(2|2n+1)$ or $G(3)$, i.e. in the two remaining   cases listed in \cite[Proposition  8.10]{KMP1}. The reason is that in these cases  $0$ is a $\h^\natural$-weight of $\g_{-1/2}$.
\end{remark}

We call a $\Wu$-module {\sl ordinary} if it is $f$-twisted with $f$ the identity on $\Wu$.
\noindent Assume that $M$ is an ordinary $L\langle th^R\rangle$-positive energy module for all $t$ in a subset $I$ of $\R$. 
Let $\omega$ be a conjugate linear involution  on $\Wu$ such that 
\begin{equation}\label{hRReal}
\omega(h^R)=-h^R.
\end{equation}
  Applying Proposition \ref{functor}, we find that $M^\dagger$ is an $L\langle3th^R\rangle$-positive energy $Id\langle t h^R_0\rangle$-twisted module, so, applying Proposition \ref{functor} again, we find that, if $s\in3I$, then $(M^\dagger)^\dagger$ is an $L\langle3sh^R\rangle$-positive energy $(Id\langle t h^R_0\rangle)\langle s h_0^R\rangle$-twisted module. Since
 $$(Id\langle t h^R_0\rangle)\langle s h_0^R\rangle=\omega^{-1}e^{-4\pi\sqrt{-}1sh^R_0}\omega^{-1}e^{-4\pi\sqrt{-}1th^R_0}\omega^2=e^{-4\pi\sqrt{-}1(s+t)h^R_0},
 $$
  we find that, choosing $s=\tfrac{1}{4}-t$, $(M^\dagger)^\dagger$ is a $\s_R$-twisted module.

 There is a canonical linear embedding $M\to (M^\dagger)^\dagger$ given by $m\mapsto F_m$ with $F_m(\lambda)=\ov{\lambda(m)}$. We define a structure of  $\s_R$-twisted $\Wu$-module on $M$ via the fields $Y^R(a,z)$ defined by
  $$
F_{ Y^R(a,z)m}=Y^{(sh^R)}(a,z)(F_m).
$$
 If $a\in\Wu$, write $Y^R(a,z)=\sum_{n\in\Z} a^R_nz^{-n-\D_a}$. If $a\in\g$ is an eigenvector for $ad(h^R)$, write $\eta_a$ for the corresponding eigenvalue. By \eqref{eigenhR}
 \begin{equation}
 \eta_a\in \{\pm2,0\}\text{ if $a\in\g^\natural$}, \ \eta_a\in \{\pm1\}\text{ if $a\in\g_{-1/2}$}.
 \end{equation}
 Recall that there is an embedding $V^{\be_k}(\g^\natural)\to\Wu$, where $\be_k$ is the cocycle given explicitly in \cite[Theorem 5.1]{KW1}.
 
 \begin{lemma}Fix $n\in\ZZ$ and $m\in M$. If $a\in \g^\natural$ and $v\in\g_{-1/2}$, then 
\begin{align}
&(J^{\{ a \}})^{R}_n m=e^{-\frac{\pi}{4}\sqrt{-}1\eta_a}J^{\{a\}}_{n+1/2\eta_a}m+\half\d_{n,0}\beta_k(h^R,a)m,\label{JAR}\\
&(G^{\{ v \}})^{R}_n m  = e^{-\frac{\pi}{4}\sqrt{-}1\eta_v} G^{\{v\}}_{n+1/2\eta_v}m,\label{GAR}\\
&L^{R}_n m  =L_nm+\half J^{\{h^R\}}_nm+\d_{n,0}\tfrac{1}{8}\be_k(h^R,h^R)m\label{LAR}.
\end{align}
\end{lemma}
\begin{proof}To simplify notation, in this proof we set $h=h^R$.
 Observe that
 $g(L\langle t h\rangle)=\omega(L\langle t h\rangle)=L-t\partial h$. It follows that,
if $a\in \g^\natural$ and $v\in \g_{-1/2}$, then 
$$
 \begin{aligned}
& g(L\langle t h\rangle)(1,t)g((J^{\{ a \}}))=(L_1-t\partial h_1)g(J^{\{ a \}})=(L_1-t\partial h_1)e^{-\pi\sqrt{-}1\D_a(t)}J^{\{\omega(a)\}}\\
&=2te^{-\pi\sqrt{-}1\D_a(t)}\beta_k(h,\omega(a))\vac,
 \\
  &g(L\langle t h\rangle)(1,t)g(G^{\{ v \}})=(L_1-t\partial h_1)e^{-\pi\sqrt{-}1(\D_v(t)+1/2)}G^{\{\omega( v) \}}=0,\\
   & g(L\langle t h\rangle)(1,t)g(L)=(L_1-t\partial h_1)L=2t h.
 \end{aligned}
 $$
It follows from \eqref{undagger} that, if $a\in g^\natural$ and $n\in\ZZ$,
 $$
 \begin{aligned}
& ((J^{\{ a \}})^{(sh)}_nF_m)(\lambda)= ((J^{\{ a \}})^{(sh)}(n+3s\eta_a,3s)F_m)(\lambda)\\
&= F_m(g(J^{\{ a \}})^{(th)}(-n-3s\eta_a,s)\lambda)\\
&+2se^{\pi\sqrt{-}1\D_a(s)}\ov{\beta_k(h,\omega(a))}F_m(\vac(-n-3s\eta_a,s)\lambda)\\
&= F_m(e^{-\pi\sqrt{-}1\D_a(s)}(J^{\{\omega(a)\}})^{(th)}(-n-3s\eta_a,s)\lambda)\\
&+\d_{n+3s\eta_a,0}2s\ov{e^{-\pi\sqrt{-}1\D_a(s)}\beta_k(h,\omega(a))\lambda(m)}\\
&= \ov{ (e^{-\pi\sqrt{-}1\D_a(s)}(J^{\{\omega(a)\}})^{(th)}(-n-(2s+3t)\eta_a,3t)\lambda)(m)}\\
&+\d_{n+3s\eta_a,0}2s\ov{e^{-\pi\sqrt{-}1\D_a(s)}\beta_k(h,\omega(a))\lambda(m)}\\
&= \ov{ (e^{-\pi\sqrt{-}1\D_a(s)}\lambda(e^{-\pi\sqrt{-}1\D_{\omega(a)}(t)}J^{\{a\}}(n+(2s+3t)\eta_a,t)m)}\\
&+\ov{2te^{-\pi\sqrt{-}1\D_a(s)}\lambda(e^{-\pi\sqrt{-}1\D_{\omega(a)}(t)}\beta_k(h,a)\vac(n+(2s+3t)\eta_a,t)m)}\\
&+\d_{n+3s\eta_a,0}2s\ov{e^{-\pi\sqrt{-}1\D_a(s)}\beta_k(h,\omega(a))\lambda(m)}\\
&= \ov{ (e^{-\pi\sqrt{-}1(\D_a(s)-\D_{\omega(a)}(t))}\lambda(J^{\{a\}}_{n+2(s+t)\eta_a}m)}\\
&+\ov{\left(\d_{n+(2s+3t)\eta_a,0}2te^{-\pi\sqrt{-}1(\D_a(s)-\D_{\omega(a)}(t))}\ov{\beta_k(h,a)}+\d_{n+3s\eta_a,0}2se^{-\pi\sqrt{-}1\D_a(s)}\beta_k(h,\omega(a))\right)\lambda(m)}.
 \end{aligned}
 $$
Set now  $s=\frac{1}{4}-t$ and observe that $\be_k(h,a)=\be_k(h,\omega(a))=0$ unless $\eta_a=0$.   Since $\omega$ is an automorphism of $\Wu$, it follows that  $\be_k(\omega(a),\omega(b))=\ov{(a,b)}$. Moreover,
 $$
 \D_a(s)-\D_{\omega(a)}(t)=1-(\frac{1}{4}-t)\eta_a-1-t\eta_a=-\frac{\eta_a}{4}.
 $$
 Putting together all these observations, our formula reduces to
  $$
 \begin{aligned}
& ((J^{\{ a \}})^{(sh)}_nF_m)(\lambda)
= e^{-\frac{\pi}{4}\sqrt{-}1\eta_a}F_{ J^{\{a\}}_{n+1/2\eta_a}m}(\lambda)+\half\d_{n,0}\beta_k(h,a)F_m(\lambda),
 \end{aligned}
 $$
 which is \eqref{JAR}.
 Similarly one obtains, for $n\in\ZZ$ and $v\in\g_{-1/2}$, 
 $$
 \begin{aligned}
& ((G^{\{ v \}})^{(sh)}_nF_m)(\lambda)= F_m(e^{-\pi\sqrt{-}1(\D_v(s)+1/2)}G^{\{\omega( v) \}}(-n-3s\eta_v,s)\lambda)\\
&= \ov{ e^{-\pi\sqrt{-}1(\D_v(s)+1/2)}G^{\{\omega( v) \}}(-n-3s\eta_v,s)\lambda(m)}\\
&= \ov{ e^{-\pi\sqrt{-}1(\D_v(s)+1/2)}\lambda(e^{-\pi\sqrt{-}1(\D_{\omega(v)}(t)+1/2)}G^{\{v\}}(n+(2s+3t)\eta_v,t)m)}\\
&= \ov{ (e^{-\pi\sqrt{-}1(\D_v(s)-\D_{\omega(v)}(t))}\lambda(G^{\{v\}}_{n+2(s+t)\eta_v}m)}= e^{-\frac{\pi}{4}\sqrt{-}1\eta_v} F_{ G^{\{v\}}_{n+1/2\eta_v}m}(\lambda),
 \end{aligned}
 $$
 which gives \eqref{GAR}.
 Finally, if $n\in\ZZ$,
 $$
  \begin{aligned}
 &(L^{(sh)}_nF_m)(\lambda)=(L^{(sh)}(n,3s)F_m)(\lambda)\\
 &=F_m(\left(L^{(th)}(-n,s)+2s(J^{\{h\}})^{(th)}(-n,s)+\d_{n,0}2s^2\be_k(h,h)\right)\lambda)\\
  &=\ov{\left(L^{(th)}(-n,s)+2s(J^{\{h\}})^{(th)}(-n,s)+\d_{n,0}2s^2\be_k(h,h)\right)\lambda)(m)}\\
  &=\ov{\lambda((L_n+2tJ^{\{h\}}_n+\d_{n,0}2t^2\be_k(h,h))m)}\\
  &+\ov{2s\lambda((J^{\{h\}}_n+\d_{n,0}2t\be_k(h,h))m)} +\d_{n,0}2s^2\ov{\be_k(h,h)\lambda(m)}\\
  &=\ov{\lambda(\left(L_n+2(t+s)J^{\{h\}}_n+\d_{n,0}2(t+s)^2\be_k(h,h)\right)m)}  .
  \end{aligned} 
  $$
  Since $\ov{\be_k(h,h)}=\be_k(\omega(h),\omega(h))=\be_k(h,h)$, we obtain \eqref{LAR}.
\end{proof}

\begin{remark}
Clearly \eqref{JAR}, \eqref{GAR}, and \eqref{LAR} can be inverted to obtain
\begin{align}
&J^{\{a\}}_{n}m=e^{\frac{\pi}{4}\sqrt{-}1\eta_a}((J^{\{ a \}})^{R}_{n-1/2\eta_a} m-\half\d_{n,1/2\eta_a}\beta_k(h^R,a)m),\label{JAOR}\\
&  G^{\{v\}}_{n}m= e^{\frac{\pi}{4}\sqrt{-}1\eta_v}(G^{\{ v \}})^{R}_{n-1/2\eta_v} m ,\label{GAOR}\\
& L_nm=L^{R}_n m -\half (J^{\{h^R\}})^R_nm-\d_{n,0}\tfrac{3}{8}\be_k(h^R,h^R)m\label{LAOR}.
\end{align}
\end{remark}

%Note that $u\in\mathfrak n_-^\natural$ instead of  $\mathfrak n_+^\natural$ in the last equality above.
 
 \section{Spectral flow and unitary highest weight modules}\label{5}
We now assume that $k$ is in the unitary range. This implies that  there is a conjugate linear involution $\omega$ on $\Wu$ and  a semi-positive definite  $\omega$-invariant Hermitian form $(\cdot,\cdot)$ on $\Wu$. In other words the simple $W$-algebra $\Ws$ is a unitary vertex algebra. We normalize this  form by setting $(\vac,\vac)=1$.
 Note that $\omega_{|\g^\natural}$ must be the conjugation with respect to a compact form. This implies in particular that \eqref{hRReal} holds.
  
  Remark that
$\D^{NS}_+=\{\a\in\D^\natural\mid \a(h^R)<0\}\cup \{\a\in\D^\natural_+\mid \a(h^R)=0\}$ is a set of positive roots for $\g^\natural$. Let 
$\mathfrak n^{NS}_-\oplus \h^\natural \oplus \mathfrak n^{NS}_+$ be the corresponding triangular decomposition of $\g^\natural$. 
By  a {\sl highest weight ordinary  $\Wu$-module of highest weight $(\nu,\ell)$ }we mean an ordinary $\Wu$-module $M$ generated by a vector $v_{\nu,\ell}$ such that
\begin{align*}
&J^{\{h\}}_0v_{\nu,\ell}=\nu(h)v_{\nu, \ell}\ {\rm  for} \  h\in\h^\natural,\quad L_0v_{\nu,\ell}=\ell v_{\nu,\ell},\\
&J^{\{u\}}_nv_{\nu,\ell}=G^{\{v\}}_nv_{\nu,\ell}=L_nv_{\nu,\ell}=0 \ {\rm  for} \  n>0, u\in \g^\natural,\ v\in\g_{-1/2}, \\
&J^{\{u\}}_0v_{\nu,\ell}=0\ {\rm  for }\ u\in\n^{NS}_+.
\end{align*}

Let $M$ be a unitary ordinary highest weight module. This implies that the minimal energy space $M_0$ is finite dimensional. Recall that $M$ is linearly spanned by monomials
$$
\begin{aligned}\label{formulona}
&J^{\{a_{1}\}}_{i_1}\cdots J^{\{a_{t}\}}_{i_t}G^{\{v_{1}\}}_{j_1}\cdots G^{\{v_{s}\}}_{j_s}L_{k_1}\cdots L_{k_r}m,
\end{aligned}
 $$
 with $m\in M_0$, $i_q,k_u\in \Z_{<0}$ and $j_h\in\
 \half+\Z_{<0}$.
We can choose $a_i, v_i$ and $m$ to be eigenvectors for the action of $h_0$ and let $\gamma_{a_i}$, $\gamma_{v_i}$ and $\l$ be the corresponding eigenvalues. Since $M_0$ is finite dimensional, $\l$ is bounded below. 
The eigenvalue for $L\langle t h\rangle_0=L_0-th_0$ on these monomials is
$$
-\sum i_q-\sum j_h-\sum k_u-\sum t\gamma_{a_q}-\sum t\gamma_{v_h}+\ell-t\la.
$$
Since the eigenvalues of $ad(h^R)$ on $\g^\natural$ are in $\{\pm2,0\}$, the eigenvalues of  $ad(h^R)$ on $\g_{-1/2}$ are in $\{\pm1\}$, we see that, choosing $|t|<\half$, 
$$
-\sum i_q-\sum j_h-\sum k_u-\sum t\gamma_{a_q}-\sum t\gamma_{v_h}+\ell-t\la\ge\ell-t\la
$$
has a minimum value $s_0$. Let 
 $M_j=\{m\in M\mid (L_0-th_0)m=(j-s_0)m\}$. The grading 
$$
M=\oplus_j M_j
$$
turns $M$ into an $L\langle t h\rangle$-positive energy module.
 Recall from \cite{KMPR} that a $\s_R$-twisted highest weight module is a module  with a cyclic vector $m$ such that 
\begin{align}
  &(J^{\{ a \}})^{M}_n m=(G^{\{ v \}})^{M}_n m=L^{M}_n m=
  0 \text{ if $n>0$ },\label{m3}\\
 &(J^{\{ a \}})^{M}_0 m=
  0 \text{ if 
     $a \in \fn_0 (\sigma_R)_+$},\label{m4}\\
     &(G^{\{ v \}})^{M}_0 m  =
  0 \text{ if 
     $v \in \fn_{-1/2} (\sigma_R)'_+$},\label{m5}
\end{align}
where $\fn_0 (\sigma_R)_+=\sum_{\a\in\D^\natural_+}\g^\natural_\a$ and $\fn_{-1/2} (\sigma_R)'_+=\sum_{\eta\in\ov \D_{1/2}^+}(\g_{-1/2})_\eta$ with the sets $\D^\natural_+$ and $\ov \D_{1/2}^+$  explicitly described in Section 6 of \cite{KMPR}. Here, for a $\h^\natural$-stable  space $\mathfrak r$ and a weight $\eta\in(\h^\natural)^*$, we denote by $\mathfrak r_\eta$ the corresponding weight space. Note that, if $\g\ne psl(2|2)$, there are two choices for the set $\overline\D^+_{1/2}$. 

Recall that we chose $\rho_R$ from Table \ref{thetahalfisnotroot}. 
The choice of $\rho_R$ is equivalent to choosing a set $\overline\D^+_{1/2}$ in \eqref{m5}. Indeed
\begin{equation}\label{roR}
\rho_R=\half\sum_{\eta\in \overline\D^+_{1/2}}\dim(\g_{-1/2})_\eta\eta=\half\!\!\!\!\!\!\!\!\sum_{(\g_{-1/2})_\eta\subset \fn_{-1/2} (\sigma_R)'_+}\!\!\!\!\!\!\!\!\!\!\!\!\dim(\g_{-1/2})_\eta\eta.
\end{equation}
 \begin{proposition}\label{Hwumaphwu}Let $M$ be a unitary highest weight ordinary $\Wu$-module of highest weight $(\nu,\ell)$.  
Choose $h^R$ as in \eqref{hoR} with $\rho_R$ chosen from Table \ref{thetahalfisnotroot}. Then 
\begin{itemize}
 \item[(a)] In the $\s_R$-twisted module  $(Y^R,M)$ the vector $v_{\nu,\ell}$ generates a $\s_R$-twisted highest weight module of highest weight $(\nu^R,\ell^R)$ with
\begin{equation}\label{fr}
 \nu^R=\nu+M_i(k)\rho_R,\ \ell^R=\ell +\frac{2}{(\theta^\natural_i|\theta^\natural_i) }(\nu|\rho_R)+\frac{1}{(\theta^\natural_i|\theta^\natural_i) }M_i(k)(\rho_R|\rho_R),
\end{equation}
where $\rho_R=\omega^i_j$ as in Table \ref{thetahalfisnotroot}. 
 \item[(b)]   If $M$ is irreducible, then $(M^\dagger)^\dagger$ is the irreducible    $\s_R$-twisted highest weight mo\-dule
 with highest weight $(\nu^R,\ell^R)$.\end{itemize}
 \end{proposition}
 \begin{proof}
 If $n\ge2$ and $a\in \g^\natural$, since $\eta_a\in\{\pm2,0\}$,
 by \eqref{JAR}, $(J^{\{a\}})^R_n v_{\nu,\ell}=0$. If $n=1$ and  $\eta_a\ge 0$, by \eqref{JAR}, $(J^{\{a\}})^R_1 v_{\nu,\ell}=0$. If $\eta_a=-2$, then $a\in \mathfrak n(R)_+$ and, by \eqref{JAR}, $(J^{\{a\}})^R_1 v_{\nu,\ell}=-\sqrt{-1}J^{\{a\}}_{0}v_{\nu,\ell}=0$.
 Observe that, if $a\in \mathfrak n_+$ then $\eta_a\ge 0$. If $n=0$ and $\eta_a=2$, then $(J^{\{a\}})^R_0 v_{\nu,\ell}=-\sqrt{-1}J^{\{a\}}_{1}v_{\nu,\ell}=0$. If $n=0$ and $\eta_a=0$ then $a\in \mathfrak n(R)_+$, so $(J^{\{a\}})^R_0 v_{\nu,\ell}=-\sqrt{-1}J^{\{a\}}_{0}v_{\nu,\ell}=0$.
 
 If $n>0$ and $v\in \g_{-1/2}$, since $\eta_v\in\{\pm1\}$, by \eqref{GAR}, $(G^{\{v\}})^R_n v_{\nu,\ell}=e^{-\frac{\pi}{4}\sqrt{-}1\eta_v} G^{\{v\}}_{n+1/2\eta_v} v_{\nu,\ell}=0$.
 If $n=0$ and $v \in \fn_{-1/2} (\sigma)'_+$ then $\eta_v=1$ so, by \eqref{GAR}, $(G^{\{v\}})^R_0 v_{\nu,\ell}=e^{-\frac{\pi}{4}\sqrt{-}1} G^{\{v\}}_{1/2} v_{\nu,\ell}=0$.
 
 If $n>0$ then, by \eqref{LAR}, 
 $$L^{R}_n v_{\nu,\ell}  =L_nv_{\nu,\ell}+\half J^{\{h^R\}}_nv_{\nu,\ell}+\d_{n,0}\tfrac{1}{8}\be_k(h^R,h^R)v_{\nu,\ell}=0.
 $$
 
 If $a\in \h^\natural$, then, by \eqref{JAR},
 $$
 (J^{\{ a \}})^{R}_0 v_{\nu,\ell}=J^{\{a\}}_{0}v_{\nu,\ell}+\half\beta_k(h^R,a)v_{\nu,\ell}=(\nu(a)+\half\beta_k(h^R,a))v_{\nu,\ell}.
 $$
 By \cite[(7.22)]{KMP1} if $h^R\in\g^\natural_i$ and $a\in\g^\natural_j$, we have
 $$
\be_k(h^R,a)= \d_{i,j} M_i(k)\frac{(\theta_i|\theta_i)}{2}(h^R|a)=\d_{i,j} 2M_i(k)\rho_R(a).
 $$
 By \eqref{LAR},
 $$
 L^{R}_0 v_{\nu,\ell}=(\ell +\half \nu(h^R)+\tfrac{1}{8}\be_k(h^R,h^R)) v_{\nu,\ell}=(\ell +\frac{2}{(\theta^\natural_i|\theta^\natural_i) }(\nu|\rho_R)+\frac{1}{(\theta^\natural_i|\theta^\natural_i) }M_i(k)(\rho_R|\rho_R)) v_{\nu,\ell}.
 $$
 
 Assume now that $M$ is irreducible. If $N$ is a  proper submodule of $(M^\dagger)^\dagger$, then it is graded by $L\langle9th\rangle_0$. It follows that $N^\perp$ is generated by $M_0^\dagger$, hence $(N^\perp)^\perp$ is a graded proper submodule of $M$ such that $(N^\perp)^\perp\cap M_0=\{0\}$. Thus $N=\{0\}$ and $(M^\dagger)^\dagger$ is irreducible. 
 \end{proof}
Let $P^k_+(R)\subset (\h^\natural)^*$ be the set of dominant integral weights for $\Delta^\natural_+$ such that $\nu((\theta^\natural_i)^\vee)\le M_i(k)$ for all $i$.
Let also $P^k_+(NS)\subset (\h^\natural)^*$ be the set of dominant integral weights for $\D^{NS}_+$ such that $\nu((\theta^\natural_i(NS))^\vee)\le M_i(k)$ for all $i$, where  $\theta^\natural_i(NS)$ is the highest root of $\g^\natural_i$ in $\D^{NS}_+$.
Recall that  the ordinary highest weight module $L(\nu,\ell)$ can be unitary only if $\nu\in P^k_+(NS)$ and $\ell\ge  A^{NS}(k,\nu)$ with $ A^{NS}(k,\nu)$ given by \cite[(8.11)]{KMP1}.
%\begin{equation}\label{Aknu}
 %A^{NS}(k,\nu)=\frac{(\nu|\nu+2\rho^\natural(NS))}{2(k+h^\vee)}+%\frac{(\xi|\nu)}{k+h^\vee}((\xi|\nu)-k-1).
 %\end{equation}
 
Similarly, a $\s_R$-twisted highest weight module $L^R(\nu,\ell)$ can be unitary only if $\nu\in P^k_+(R)$ and $\ell\ge A^R(k,\nu)$ with
 $ A^R(k,\nu)$ given by \cite[(6.31)]{KMPR}. Explicit expressions for both $A^{NS}(k,\nu)$ and $A^R(k,\nu)$ are given case by case in \cite[Section 12]{KMP1} and \cite[Section 10]{KMPR} respectively.
 %$$
% =\tfrac{1}{2(k+h^\vee)}\left((\nu|\nu+2\rho^\natural)    -\tfrac{1}{2} p(k)+F_{\nu}(\eta_{\min})\right).
%$$
\begin{lemma}\label{remarkable} Let $\nu\in P^k_+(NS)$  and set $ \nu^R=\nu+M_i(k)\rho_R$ (as in Proposition \ref{Hwumaphwu}) with $\rho_R$ chosen from Table \ref{thetahalfisnotroot}.
Then $\nu^R\in P^k_+(R)$ and
\begin{equation}\label{mf}
A^{NS}(k,\nu) +\frac{2}{(\theta^\natural_i|\theta^\natural_i) }(\nu|\rho_R)+\frac{1}{(\theta^\natural_i|\theta^\natural_i) }M_i(k)(\rho_R|\rho_R)=A^R(k,\nu^R).
\end{equation}
\end{lemma}
\begin{proof}Let $\widehat \g^\natural=\left(\C[t,t^{-1}]\otimes \g^\natural\right)\oplus(\oplus_{i=0}^{r_0}\C K_i)\oplus \C \frac{d}{dt}$ be the affinization of $\g^\natural$ and set $\ha^\natural=\h^\natural\oplus(\oplus_{i=0}^{r_0}\C K_i)\oplus \C \frac{d}{dt}$. Let $(\cdot|\cdot)^\natural$ be the  invariant symmetric bilinear form on $\g^\natural$ such that $(\theta^\natural_i|\theta^\natural_i)^{\natural}=2$ for all $i$. If $\rho_R=\omega^i_j$   set, for $\l\in(\ha^\natural)^*$,
$$
t_{\rho_R}(\l) = \l+ \l( K_i)\rho_R-((\l_{|\h^\natural}|\rho_R)^\natural +\half(\rho_R|\rho_R)^\natural\l(K_i))\d.
$$
We note that $\rho_R=\omega^i_j$ with $\a_j$ a simple root for $\g^\natural$ such that, if $\theta^\natural_i=\sum_{r\in\Z_+} a_r\a_r$ then $a_j=1$. Let
$\D^\natural(j)$ denote the root subsystem of $\D^\natural$ generated by
$S\setminus\{\a_j\}$ and by $w^j_0$ the longest (w.r.t.
$S\setminus\{\a_j\}$) element of the
corresponding parabolic subgroup of $W$. Let $w_{0}$ be the longest
element of $W$ with respect to $S$. Then it is observed in \cite[Theorem D]{Compatible} that from the results of \cite{IwMa} one deduces that $t_{\rho_R}w^j_0w_0(P^+_k(R))=P^+_k(R)$. Note that $\D^{NS}_+=w^j_0w_0(\D^\natural_+)$, so if $\nu\in P^+_k(NS)$ then $\nu=w_0^jw_0(\nu')$ with  $\nu'\in P^+_k(R)$ so $\nu^R=t_{\rho_R}(\nu)=t_{\rho_R}w_0^jw_0(\nu')\in P^+_k(R)$.

We will prove \eqref{mf} by a case by case inspection using the explicit expressions for $A^{NS}(k,\nu)$ and $A^R(k,\nu)$ given in \cite[Section 12]{KMP1} and \cite[Section 10]{KMPR} respectively.

\subsection{$psl(2|2)$} In this case $\g^\natural=sl(2)$, $M_1(k)=-k-1$. We choose
$\D^\natural_+=\{\d_1-\d_2\}$ so that  
$$\rho_R=\half(\d_1-\d_2), \ h^R=-2\rho_R,\ \D^{NS}_+=\{-\d_1+\d_2\}, 
$$
and
$$
 P^+_k(NS)=\{-\frac{r}{2}(\d_1-\d_2)\mid 0\le r\le M_1(k)\},\ P^+_k(R)=\{\frac{r}{2}(\d_1-\d_2)\mid 0\le r\le M_1(k)\}.
 $$
 In this case
$$
A^{NS}(k,\nu)=\frac{r}{2},\quad
A^R(k,\nu^R)= -\frac{k+1}{4},
$$
and, indeed,
$$
A^{NS}(k,\nu) +\frac{2}{(\theta^\natural_i|\theta^\natural_i) }(\nu|\rho_R)+\frac{1}{(\theta^\natural_i|\theta^\natural_i) }M_i(k)(\rho_R|\rho_R)=\frac{r}{2}-\frac{r}{2}+\frac{1}{4}(-k-1)=A^R(k,\nu^R).
$$

\subsection{$spo(2|2r)$, $r>2$}In this case
$$\g^\natural=so(2r),\ M_1(k)=-2k-1.
$$
Assume first $\rho_R=\omega_r^1$. 
 Then
$$
P^+_k(R)=\{\nu=\sum_i m_i\e_i,\,m_i\in\half +\ZZ\text{ or }m_i\in\ZZ,\ m_1\geq\ldots\geq m_{r-1}\geq |m_r|,\ m_1+m_2\le M_1(k)\},
$$
and 
$$
P^+_k(NS)=\{\nu=\sum_i m_i\e_i,\,m_i\in\half +\ZZ\text{ or }m_i\in\ZZ,\ -|m_1|\geq\ldots\geq m_{r-1}\geq m_r,\ -m_{r-1}-m_r\le M_1(k)\}.
$$
Since
\begin{equation*}\label{1a}
A^{NS}(k,\nu)=-\frac{(\sum_{i=1}^r (m_i^2-2m_i(i-1))+m_r (2k-m_r+2)}{4 (k-r+2)},
\end{equation*}
\begin{align*}
A^R(k,\nu)&=\frac{-4\left( \sum_{i=1}^{r-1}(2 (r-i)-1)m_i +m_i^2\right)-4 k^2+2(r-4) k  +r-3}{16 (k+2-r)},
\end{align*}
and
 $$
 \frac{2}{(\theta^\natural_1|\theta^\natural_1) }(\nu|\rho_R)=\half \sum m_i,\quad\frac{1}{(\theta^\natural_1|\theta^\natural_1) }M_1(k)(\rho_R|\rho_R)=\frac{M_1(k)}{8}r,
 $$
 \eqref{mf} reads
 \begin{equation}\label{equivalent}
\begin{aligned}
&\frac{-4\left( \sum_{i=1}^{r-1}(2 (r-i)-1)(m_i+\half(-2k-1) )+(m_i+\frac{1}{2}(-2k-1))^2\right)}{16 (k+2-r)}\\
&-\frac{4 k^2-2(r-4) k  -r+3}{16 (k+2-r)}\\
&=\frac{-4(\sum_{i=1}^r (m_i^2-2m_i(i-1)+m_r (2k-m_r+2)) +8 (k-r+2) \sum m_i}{16 (k-r+2)}\\
&+\frac{2r(-2k-1) (k-r+2)}{16 (k-r+2)},
\end{aligned}
\end{equation}
or, equivalently,
\begin{equation}\label{equivalentbis}
\begin{aligned}
&\sum_{i=1}^{r-1}\left( 8 k m_i-8 i k-8 r m_i-4m_i^2+8 i m_i+8 m_i-4i-4k^2+8 k r-8 k+4r-3\right)\\
&-4 k^2+2(r-4) k  +r-3\\
&=\sum_{i=1}^{r-1}\left(8 k m_i-8 r m_i-4 m_i^2+8 i m_i+8 m_i\right)+2r(-2k-1) (k-r+2),
\end{aligned}
\end{equation}
which is readily verified.

If $\rho_R=\omega_{r-1}^1$,
then  \eqref{mf} is obtained from \eqref{equivalent} by changing $m_r$ with $-m_r$. Since \eqref{equivalent} is equivalent to \eqref{equivalentbis} and the latter equation does not depend on $m_r$, we deduce that \eqref{mf} is satisfied in this case as well.

\subsection{$D(2,1,\tfrac{m}{n})\  \text{\rm with $m,n$ coprime}$}
In this case $k=-\frac{mn}{m+n}t$ with $t\in\nat$,
$$\g^\natural=sl(2)\times sl(2),\ M_1(k)=mt-1,\ M_2(k)=nt-1.
$$
Assume first $\rho_R=\omega_1^1$. 
 Then
$$
P^+_k(R)=\{\nu=\sum_{i=1}^2 m_i\e_{i+1}\mid m_i\in \ZZ_+,\ 0\le m_i\le M_i(k)\},
$$
and 
$$
P^+_k(NS)=\{\nu=-m_1\e_2+m_2\e_3\mid m_i\in \ZZ_+,\ 0\le m_i\le M_i(k)\}.
$$
Since
\begin{equation*}
A^{NS}(k,\nu)=\frac{ (m_1-
  m_2 )^2+2t(
  m_2 m + m_1
 n)}{4 (m+n )t},
\end{equation*}
\begin{align*}
A^R(k,\nu)=\frac{(1 + m_1 + m_2)^2  + t (-m-n + mnt)}{4 (m + n) t},
\end{align*}
and, if $\nu=-m_1\e_2+m_2\e_3\in P^+_k(NS)$, 
 $$
 \frac{2}{(\theta^\natural_1|\theta^\natural_1) }(\nu|\rho_R)=-\frac{m_1}{2},\quad\frac{1}{(\theta^\natural_1|\theta^\natural_1) }M_1(k)(\rho_R|\rho_R)=\frac{M_1(k)}{4},
 $$
it follows that \eqref{mf} becomes
$$
m^2 t^2+m n t^2-m t-n t-2 m m_1 t+2 m m_2 t+m_1^2+m_2^2-2
   m_1 m_2=t (m n t-m-n)+\left(m t-m_1+m_2\right)^2,
$$
which is readily verified.

If $\rho_R=\omega_2^1$, then
$$
P^+_k(NS)=\{\nu=m_1\e_2-m_2\e_3\mid m_i\in \ZZ_+,\ 0\le m_i\le M_i(k)\},
$$
and, if $\nu=m_1\e_2-m_2\e_3\in P^+_k(NS)$, 
 $$
 \frac{2}{(\theta^\natural_2|\theta^\natural_2) }(\nu|\rho_R)=-\frac{m_2}{2},\quad\frac{1}{(\theta^\natural_2|\theta^\natural_2) }M_1(k)(\rho_R|\rho_R)=\frac{M_2(k)}{4}.
 $$
and we argue as above.

  \subsection{$F(4)$}\label{F4}
In this case 
$\g^\natural=so(7)$, $M_1(k)=-\frac{3}{2} k-1$ and $\rho_R=\omega_1^1$. Then, if we write $\nu=r_1\epsilon_1+r_2\epsilon_2+r_3\epsilon_3$ with $\epsilon_i$ as in \cite[Table 1]{KMP1}, we have
{\small
$$
P^+_k(R)=\{m_1\epsilon_1+m_2\epsilon_2+m_3\epsilon_3\mid m_i\in \ZZ_+\text{ or }m_i\in\half  \ZZ_+,\ m_1\ge m_2\ge m_3\ge0,\ m_1+m_2\le M_1(k)\}
$$}
and
{\small
$$
P^+_k(NS)=\{-m_1\epsilon_1+m_2\epsilon_2+m_3\epsilon_3\mid m_i\in \ZZ_+\text{ or }m_i\in\half  \ZZ_+,\ m_1\ge m_2\ge m_3\ge0,\ m_1+m_2\le M_1(k)\}.
$$}

Since{\small
\begin{align*}\label{1c}
A^{NS}(k,\nu)= \frac{m_1 (6-\tfrac{3}{2}k)+m_2 (3-\tfrac{3}{2}k)+m_3
   (-\tfrac{3}{2}k)+m_1^2+m_2^2+m_3^2-m_1m_2-m_1m_3-m_2m_3}{3 (3-\tfrac{3}{2}k)},
\end{align*}
\begin{align*}
A^R(k,\nu)&=-\frac{9 k^2+8 m_1^2+8 m_1 (m_2+m_3+5)+8 m_2^2-8 m_2
   m_3+32 m_2+8 m_3^2+8 m_3-4}{36 (k-2)},
   \end{align*}}
and, if $\nu=-m_1\e_1+m_2\e_2+m_3\e_3\in P^+_k(NS)$, 
 $$
 \frac{2}{(\theta^\natural_1|\theta^\natural_1) }(\nu|\rho_R)=-m_1,\quad\frac{1}{(\theta^\natural_1|\theta^\natural_1) }M_1(k)(\rho_R|\rho_R)=\frac{M_1(k)}{2},
 $$
Since $\nu^R=(M_1(k)-m_1)\e_1+m_2\e_2+m_3\e_3$, it follows that \eqref{mf} becomes
{\small
  $$
  \begin{aligned}
&4 \frac{m_1 (12-3k)+m_2 (6-3k)+m_3
   (-3k)+2m_1^2+2m_2^2+2m_3^2-2m_1m_2-2m_1m_3-2m_2m_3}{36 (2-k)}\\
   &-m_1+\half M_1(k)\\
   &=
   \frac{9 k^2+8 (M_1(k)-m_1)^2+8 (M_1(k)-m_1) (m_2+m_3+5)+8 m_2^2-8 m_2
   m_3+32 m_2+8 m_3^2+8 m_3-4}{36 (2-k)}
   \end{aligned}
 $$}
 and this formula holds.
\end{proof}

First we observe that the spectral flow maps unitary modules to unitary modules.
Recall  that a $\Wu$-module $M$ (ordinary or $\s_R$-twisted) is unitary if it admits a positive definite  Hermitian form $H$ such that
\begin{equation}
\begin{aligned}
H(m,J^{\{a\}}_nm)&=-H(J^{\{\omega(a)\}}_{-n}m,{m^\dagger}),\\
H(m,G^{\{v\}}_nm)&=H(G^{\{\omega(v)\}}_{-n}m,{m^\dagger}),\\
H(m,L_nm)&=H(L_{-n}m,{m^\dagger}).
\end{aligned}
\end{equation}

\begin{lemma}\label{sfunitary}
 $(Y_M,M)$ is an ordinary  unitary module if and only if $(Y^R,M)$ is a $\s_R$-twisted unitary module. 
\end{lemma}
\begin{proof}
First, by \eqref{JAR},
$$
\begin{aligned}
&H(m,(J^{\{ a \}})^{R}_n {m^\dagger})=e^{-\frac{\pi}{4}\sqrt{-}1\eta_a}H(m,J^{\{a\}}_{n+1/2\eta_a}{m^\dagger})+\half\d_{n,0}\beta_k(h^R,a)H(m, {m^\dagger})\\
&=-e^{-\frac{\pi}{4}\sqrt{-}1\eta_a}H(J^{\{\omega(a)\}}_{-n-1/2\eta_a}m,{m^\dagger})+\half\d_{n,0}H(\ov{\beta_k(h^R,a)}m, {m^\dagger})\\
&=-H(e^{-\frac{\pi}{4}\sqrt{-}1\eta_{\omega(a)}}J^{\{\omega(a)\}}_{-n+1/2\eta_{\omega(a)}}m,{m^\dagger})-\half\d_{n,0}H(\beta_k(h^R,\omega(a))m ,{m^\dagger})\\
&=-H((J^{\{\omega(a)\}})^R_{-n}m,{m^\dagger}).
\end{aligned}
$$
Next, by \eqref{GAR},
$$
\begin{aligned}
&H(m,(G^{\{ v \}})^{R}_n {m^\dagger})=e^{-\frac{\pi}{4}\sqrt{-}1\eta_v} H(m, G^{\{v\}}_{n+1/2\eta_v}{m^\dagger})\\
&=e^{-\frac{\pi}{4}\sqrt{-}1\eta_v}H(G^{\{\omega(v)\}}_{-n-1/2\eta_v}m,{m^\dagger})=H(e^{-\frac{\pi}{4}\sqrt{-}1\eta_{\omega(v)}}G^{\{\omega(v)\}}_{-n+1/2\eta_{\omega(v)}}m,{m^\dagger})\\
&=H((G^{\{\omega(v)\}})^R_{-n}m,{m^\dagger}).
\end{aligned}
$$
Finally, since $\be_k(h^R,h^R)\in\R$, by \eqref{LAR},
$$
\begin{aligned}
&H(m,L^{R}_n {m^\dagger})=H(m,L_n{m^\dagger})+\half H(m,J^{\{h^R\}}_n{m^\dagger})+\d_{n,0}\tfrac{1}{8}\be_k(h^R,h^R)H(m,{m^\dagger})\\
&=H(L_{-n}m,{m^\dagger})+\half H(J^{\{h^R\}}_{-n}m,{m^\dagger})+\d_{n,0}\tfrac{1}{8}H(\be_k(h^R,h^R)m,{m^\dagger})=H(L^R_{-n}m,{m^\dagger}).
\end{aligned}
$$
This proves that $(Y^R,M)$ is unitary. The reverse statement is obtained by the same argument using \eqref{JAOR}, \eqref{GAOR}, and \eqref{LAOR}.
\end{proof}

In Proposition \ref{Hwumaphwu} and Lemma \ref{remarkable} we restricted ourselves to $\rho_R$ from Table \ref{thetahalfisnotroot}. According to \cite{KMPR}, if $\g=F(4)$, one can choose also  $\rho_R=\omega^1_3$. To deal with this case we need the following result. If $\psi$ is a weight for $\g_{-1/2}$, we denote by $v_\psi$ a corresponding weight vector. 
\begin{lemma}\label{42}Let  $\g=F(4)$ and $\rho_R=\omega^1_3$. Let  $M$ be a $\s_R$-twisted highest weight module of highest weight $(\nu,\ell)$ such that \eqref{m5} holds with 
$$\ov \D_{1/2}=\{\half(\e_1+\e_2+\e_3),\half(\e_1-\e_2+\e_3),\half(\e_1+\e_2-\e_3),\half(\e_1-\e_2-\e_3)\}.
$$
Set $v=v_{\frac{1}{2}}(-\e_1+\e_2+\e_3)$. If $G^{\{v\}}_0v_{\nu,\ell}\ne0$, then  it is a highest weight vector satisfying  \eqref{m5} with 
$$\ov \D_{1/2}=\{\half(\e_1+\e_2+\e_3),\half(\e_1-\e_2+\e_3),\half(\e_1+\e_2-\e_3),\half(-\e_1+\e_2+\e_3)\}.
$$
Its   highest weight is $(\nu',\ell)$, where 
$\nu'=\nu-\omega_1^1+\omega_3^1$.
\end{lemma}

\begin{proof}
We first check that 
\begin{align}
&J^{\{x_{-\theta}\}}_1G^{\{v\}}_0v_{\nu,\ell}=0,J^{\{x_{\e_1-\e_2}\}}_0G^{\{v\}}_0v_{\nu,\ell}=0,J^{\{x_{\e_2-\e_3}\}}_0G^{\{v\}}_0v_{\nu,\ell}=0, J^{\{x_{\e_3}\}}_0G^{\{v\}}_0v_{\nu,\ell}=0\label{J1},\\
&G^{\{v\}}_0G^{\{v\}}_0v_{\nu,\ell}=0\label{G2}.
\end{align}

It is clear that $[J^{\{x_{-\theta}\}}_1,G^{\{v\}}_0]=0$ hence $J^{\{x_{-\theta}\}}_1G^{\{v\}}_0v_{\nu,\ell}=J^{\{x_{-\theta}\}}_1G^{\{v\}}_0J^{\{x_{-\theta}\}}_1v_{\nu,\ell}=0$. Similarly for the third and fourth relation in \eqref{J1}. For the second relation in \eqref{J1}, observe that
$[J^{\{x_{\e_2-\e_3}\}}_0,G^{\{v\}}_0]=G^{\{[x_{\e_2-\e_3},v]\}}_0$ and $[x_{\e_2-\e_3},v]$ has weight $\half(\e_1-\e_2+\e_3)$ so
$$
J^{\{x_{\e_1-\e_2}\}}_0G^{\{v\}}_0v_{\nu,\ell}=G^{\{[x_{\e_2-\e_3},v]\}}_0v_{\nu,\ell}+G^{\{v\}}_0J^{\{x_{\e_1-\e_2}\}}_0v_{\nu,\ell}=0.
$$

It remains to check that $G^{\{v\}}_0G^{\{v\}}_0v_{\nu,\ell}=0$. We use the formula
\begin{align*}\notag &2G^{\{v\}}_{0}G^{\{v\}}_{0}v_{\nu, \ell}=[G^{\{v\}}_{0},G^{\{v\}}_{0}]v_{\nu, \ell}=\langle v,v\rangle(-2(k+h^\vee) \ell+ (\nu|\nu+2\rho^\natural)-\tfrac{1}{2} p(k))v_{\nu,\ell}\notag\\
%6  
        &+\sum_{\a,\be}\langle[u_\alpha,v],[v,u^\be]\rangle J^{\{u_\beta\}}_0J^{\{u^\a\}}_0v_{\nu,\ell}
+  \sum_{\a,\be}\langle[u_\alpha,v],[v,u^\be]\rangle
 J^{\{u^\a\}}_0J^{\{u_\beta\}}_0v_{\nu,\ell}\\
 &=\sum_{\a,\be}\langle[u_\alpha,v],[v,u^\be]\rangle \left(J^{\{u_\beta\}}_0J^{\{u^\a\}}_0
+ J^{\{u^\a\}}_0J^{\{u_\beta\}}_0\right)v_{\nu,\ell}.
\end{align*}
Observe that the possibly nonzero contributions to the above sum  come the pairs $(\a,\be)$ of roots such that $\a-\beta=\e_1-\e_2-\e_3$. One easily checks that these pairs are exactly
$$\{(-\e_2-\e_3,-\e_1),(-\e_2,-\e_1+\e_3),(-\e_3,-\e_1+\e_2),(\e_1-\e_2,\e_3),(\e_1-\e_3,\e_2),(\e_1,\e_2+\e_3)\}
.$$
Note that each pair corresponds to  commuting root vectors, so that, 
$ J^{\{u^\a\}}_0J^{\{u_\beta\}}_0v_{\nu,\ell}=J^{\{u_\beta\}}_0J^{\{u^\a\}}_0v_{\nu,\ell}=0$, since in all cases, either $u_\beta$ or $u^\a$ is a root vector corresponding to a positive root.

It it is well-known that \eqref{J1} implies that $J_n^{\{a\}}v_{\nu,\ell}=0$ for $n>0$ and  $a\in\g^\natural$ as well as for $n=0$ and $a\in \mathfrak n_0(\s_R)_+$. 

Note that, if $w\in \mathfrak n_{-1/2}(\s_R)'_+$, then $w=[a,v]$ with $a\in U(\mathfrak n_0(\s_R)_+)$. Using the relation $[J^{\{a\}}_n,G_0^{\{v\}}]=G^{\{[a,v]\}}_n$ one obtains  that $G^{\{w\}}_0v_{\nu,\ell}=0$ for all $w\in \mathfrak n_{-1/2}(\s_R)'_+$. We now check that $G^{\{w\}}_nG_0^{\{v\}}v_{\nu,\ell}=0$ for $n>0$ and $w\in\g_{-1/2}$. 
We note that  $[J^{\{x_{-\e_2-\e_3}\}}_n,G_0^{\{v\}}]=G^{\{v_{\frac{1}{2}(-\e_1-\e_2-\e_3)}\}}_n$ and using the fact that $\g_{-1/2}$ has $(\Dp)^\natural$-lowest weight equal to $\frac{1}{2}(-\e_1-\e_2-\e_3)$, then, since $\g_{-1/2}=U(\mathfrak n_0(\s_R)_+)v_{\frac{1}{2}(-\e_1-\e_2-\e_3)}$, we can use repeatedly  $J^{\{a\}}_0$ with $a\in \mathfrak n_0(\s_R)_+$ to obtain $G^{\{w\}}_n$ for all $w\in\g_{-1/2}$.\end{proof}

Let $L^R(\nu,\ell)$ be the irreducible $\s^R$-twisted highest weight module of highest weight $(\nu,\ell)$. Let $v_{\nu,\ell}\in L^R(\nu,\ell)$ be a highest weight vector. Recall that this means that $v_{\nu,\ell}$ is a cyclic vector in $L^R(\nu,\ell)$ that satisfies \eqref{m3}, \eqref{m4}, \eqref{m5}.
 In the following result we use spectral flow to provide a proof of  \cite[Theorem 9.17]{KMPR} that does not rely on Conjecture 9.11 of \cite{KMPR}.

Recall that a weight $\nu\in P^+_k$ is said to be {\it Ramond extremal} (w.r.t. $\rho_R$) if $\nu-\rho_R\notin  P^+_k$ or $\nu-\rho_R$ is extremal (see \cite[(9.3)]{KMPR} and \cite[Definition 8.7]{KMP1}). 
\begin{theorem}\label{sfnomextremal}  If $\ell\ge A^R(k,\mu)$, $k$ is in the unitary range, and $\mu\in P^+_k(R)$  is  not Ramond extremal, then $L^R(\mu,\ell)$  is a unitary $\si_R$-twisted $W^k_{\min}(\g)$--module.
\end{theorem}
\begin{proof}
If $\rho_R=\omega^i_j$ as in Table \ref{thetahalfisnotroot}, set
$$
 \nu=\mu-M_i(k)\rho_R,\ \ell_0=\ell -\frac{2}{(\theta^\natural_i|\theta^\natural_i) }(\mu|\rho_R)+\frac{1}{(\theta^\natural_i|\theta^\natural_i) }M_i(k)(\rho_R|\rho_R),
 $$
 so that
 $\mu=\nu^R$ and $\ell=\ell_0^R$ (cf. \eqref{fr}). 
 We claim that $\nu$ is not an extremal weight in the Neveau-Schwarz sector. Indeed choose $\hat \ell\gg A^R(k,\mu)$. By \cite[Theorem 7.5]{KMPR}, $L^R(\mu,\hat \ell)$  is a unitary $\si_R$-twisted $W^k_{\min}(\g)$--module. By Lemma \ref{sfunitary}, since,  by Proposition \ref{Hwumaphwu},  $(L(\nu,\hat\ell_0)^\dagger)^\dagger=L^R(\mu,\hat \ell)$, we see that $L(\nu,\hat\ell_0)$ is unitary. 
 Since, by Lemma \ref{remarkable}, $\hat\ell_0\gg A^{NS}(k,\nu)$, combining  Proposition 8.5 and Proposition 8.8 of \cite{KMP1}, we deduce that $\nu$ is not extremal, as claimed.
 
Since $\nu$ is not extremal, Theorem   11.1 of \cite{KMP1} now implies that $L(\nu,\ell_0)$ is unitary for all $\ell_0\ge A^{NS}(k,\nu)$,  thus,  by Lemma \ref{remarkable} and Lemma \ref{sfunitary}, $(L(\nu,\ell_0)^\dagger)^\dagger=L^R(\mu,\ell)$ is unitary for all $\ell\ge A^R(k,\mu)$.

Assume now $\g=F(4)$ and $\rho_R=\omega_3^1$. Set $\mu'=\mu+\omega_1^1-\omega_3^1$. We first prove that  $\mu'$ is in $P^+_k$ and not Ramond extremal w.r.t. 
$\rho'_R=\omega_1^1$. Since  $\mu\in P^+_k$ and $\omega_1^1-\omega_3^1=\half(\e_1-\e_2-\e_3)$, we have $(\mu',\theta^\vee)\leq k$. 
Moreover, by assumption, $\mu-\rho_R\in  P^+_k$, hence $\mu-\rho_R+\rho'_R$ is a dominant integral weight.

We prove that if $\ell>A^R(k,\mu)$, then $G^{\{v\}}_0v_{\mu',\ell}\ne0$.
A direct computation using the explicit expression given in \S \ref{F4} shows that $A^R(k,\mu)=A^R(k,\mu')$. 
%By the first part of the proof, we have that highest weight module (w.r.t. to the set of positive roots attached to $\rho'_R$) $L^R(\mu',\ell)$ is unitary. In particular, s
From   Proposition 6.6, (1) in \cite{KMPR} it follows that 
\begin{equation}\label{norma} ||G^{\{v\}}_0v_{\mu',\ell}||^2=-2(k+h^\vee)\langle \phi(v),v\rangle \left( \ell- A(k,\mu')\right).\end{equation}
Here $\langle a , b\rangle=(e|[a,b])$ is the  bilinear form on $\g_{-\frac{1}{2}}$ defined e.g. in \cite[(3.4)]{KMPR}). Since $\ell>A^R(k,\mu')$,  we have  $||G^{\{v\}}_0v_{\mu',\ell}||^2\ne 0$. By Lemma \ref{42}, $G^{\{v\}}_0v_{\mu',\ell}$ is a highest weight vector, hence the highest weight module $L^R(\mu',\ell)$ (w.r.t. $\rho_R=\omega_1^1$) is the irreducible highest weight module w.r.t. $\rho_R=\omega_3^1$ of highest weight $(\mu,\ell)$.   By the first part of the proof, this module is unitary. 

Finally, if $\ell=A^R(k,\mu)$ we use the limiting argument given e.g. in \cite[Theorem 11.1]{KMP1}, \cite[Theorem  9.17]{KMPR} to conclude that $L^R(\mu,A(k,\mu))$ is unitary as well.
\end{proof}

The next result discusses the extremal representations.

\begin{theorem}\label{fromNStoR}
 The extremal representations $L(\nu,A^{NS}(k,\nu))$ are all unitary if and only if the Ramond extremal representations  $L^R(\mu,A^R(k,\mu))$ are all unitary.
\end{theorem}
\begin{proof} Assume that $\rho^R$ is as given in Table \ref{thetahalfisnotroot}.
Suppose that $L(\nu,A^{NS}(k,\nu))$ is unitary for all extremal $\nu$. Then, by Theorem   11.1 of \cite{KMP1},   $L(\nu,A^{NS}(k,\nu))$ is unitary for all $\nu\in P^+_k(NS)$. Thus, Proposition \ref{Hwumaphwu} and Lemma \ref{sfunitary} imply that $L^R(\nu^R,A^{R}(k,\nu^R))$ are all unitary. Since the map $\nu\mapsto \nu^R$ is a bijection between $P^+_k(NS)$ and $P^+_k(R)$, we deduce that $L^R(\mu,A^{R}(k,\mu)$ is unitary for all $\mu\in P^+_k(R)$, hence, in particular,  $L^R(\mu,A^{R}(k,\mu)$ is unitary  for all Ramond extremal weights.

We have proven that, if the extremal representations $L(\nu,A^{NS}(k,\nu))$ are all unitary then the Ramond extremal representations $L^R(\mu,A^R(k,\mu))$ are all unitary. 
Now we discuss the missing case, when $\g=F(4)$ and $\rho_R=\omega_3^1$. By the first part of the proof if $\mu\in P^+_k$, then the highest weight 
(w.r.t. $\rho'_R:=\omega_1^1$) module $L^R(\mu, A^R(k,\mu))$  is unitary. By \eqref{norma}, $||G^{\{v_{\frac{1}{2}(-\e_1+\e_2+\e_3)}\}}_0v_{\mu,A^R(k,\mu)}||=0$, hence
$G^{\{v_{\frac{1}{2}(-\e_1+\e_2+\e_3)}\}}_0v_{\mu,A^R(k,\mu)}=0$, so that $L^R(\mu, A^R(k,\mu))$ is a (unitary) highest weight module w.r.t. $\rho_R=\omega_3^1$ too. 

The converse statement is  deduced by reversing the argument and using Theorem \ref{sfnomextremal} above instead of Theorem 11.1 of \cite{KMP1}.
\end{proof}

\begin{remark}
It is proven in \cite{Gunaydin} that, if $\g$ is $D(2,1;a)$, then all the extremal representations (called massless there) in Neveu-Schwarz sector are unitary. A detailed proof of this fact can be found in \cite{KMPN4}. Theorem \ref{fromNStoR} now implies that all Ramond extremal representations are unitary, a fact already observed in \cite{Gunaydin}.
\end{remark}

\begin{remark} As a result of the present paper, in order to complete
the classification of unitary irreducible highest weight modules
over unitary vertex algebras $\Ws$, it remains to prove our unitarity conjectures in \cite{KMP1} and \cite{KMPR}
for extremal modules in cases $\g=spo(2|m)$ for $m>4$, $F(4)$, and $G(3)$ in the Neveu-Schwarz sector, and in cases $\g=spo(2|2m+1)$ for $m>1$, and $G(3)$ in the Ramond sector.
\end{remark}
\section*{Acknowledgments}
The authors would like to thank Dra\v{z}en Adamovi\'c  and Haisheng Li for important correspondence. The authors would like to thank the referee for his valuable suggestions on the exposition. 
Victor Kac is partially supported by Simons Travel Grant. Pierluigi M\"oseneder Frajria and Paolo Papi  are partially supported by the PRIN project 2022S8SSW2 - Algebraic and geometric aspects of Lie theory - CUP B53D2300942 0006, a project cofinanced
by European Union - Next Generation EU fund.

%\bibliographystyle{Victor}
%\bibliography{W} 

\vskip5pt
    \footnotesize{
\noindent{\bf V.K.}: Department of Mathematics, MIT, 77
Mass. Ave, Cambridge, MA 02139;\newline
{\tt kac@math.mit.edu}
\vskip5pt
\noindent{\bf P.M-F.}: Politecnico di Milano, Polo regionale di Como,
Via Anzani 42, 22100, Como, Italy;\newline {\tt pierluigi.moseneder@polimi.it}
\vskip5pt
\noindent{\bf P.P.}: Dipartimento di Matematica, Sapienza Universit\`a di Roma, P.le A. Moro 2,
00185, Roma, Italy;\newline {\tt papi@mat.uniroma1.it}, Corresponding author
}

\end{document}